\documentclass[a4paper,11pt]{article}

\usepackage{amssymb,amsmath,amsthm,times,mathrsfs,fullpage}
\usepackage{tikz,multirow}
\usetikzlibrary{shadows}
\usetikzlibrary{scopes,shapes,snakes,arrows}
\usepackage[pdftex]{hyperref}
\usepackage[skip= 0pt,textfont= it]{caption}
\hypersetup{plainpages= True,
	pdfstartview= FitH,
	bookmarksopen= true,
	pdfpagemode= none,
	colorlinks= true,
	linkcolor= blue,
	citecolor= blue}
\def\pf{\noindent \emph{Proof.}\ }
\def\qed{{\quad\rule{1mm}{3mm}\,}}
\def\dd#1{\,\mathrm{d}{#1}}
\def\le{\leqslant}
\def\ge{\geqslant}
\def\JS{\mathscr{J\!\!S}}
\def\ve{\varepsilon}
\def\tr#1{\lfloor{#1}\rfloor}
\def\lpa#1{\bigl({#1}\bigr)}
\def\Lpa#1{\Bigl({#1}\Bigr)}
\def\LLpa#1{\Biggl({#1}\Biggr)}
\def\llpa#1{\biggl({#1}\biggr)}

\begin{document}

\newtheorem{thm}{Theorem}
\newtheorem{cor}{Corollary}
\newtheorem{lmm}{Lemma}
\newtheorem{conj}{Conjecture}
\newtheorem{prop}{Proposition}
\newtheorem{Def}{Definition}
\theoremstyle{remark}\newtheorem{Rem}{Remark}

\graphicspath{{./dst_profile_figs/}}

\title{\textbf{Node Profiles of Symmetric Digital Search
Trees: Concentration Properties}}
\author{Michael Drmota\\
    Institute for Discrete Mathematics and Geometry\\
    Vienna University of Technology\\
    1040 Vienna\\
    Austria \and
Michael Fuchs\thanks{Partially supported by MOST under the grants
MOST-104-2923-M-009-006-MY3 and MOST-109-2115-M-004-003-MY2.}\\
    Department of Applied Mathematics\\
    National Chiao Tung University\\
    Hsinchu 300\\
    Taiwan \and
Hsien-Kuei Hwang\thanks{Partially supported by an Investigator Award
from Academia Sinica under the Grant AS-IA-104-M03.}\\
    Institute of Statistical Science\\
    Academia Sinica\\
    Taipei 115\\
    Taiwan  \and
Ralph Neininger\\
    Institute for Mathematics\\
    Goethe University\\
    60054 Frankfurt a.M.\\
    Germany}
\date{\today}
\maketitle

\begin{abstract}

We give a detailed asymptotic analysis of the profiles of random
symmetric digital search trees, which are in close connection with
the performance of the search complexity of random queries in such
trees. While the expected profiles have been analyzed for several
decades, the analysis of the variance turns out to be very difficult
and challenging, and requires the combination of several different
analytic techniques, including Mellin and Laplace transforms,
analytic de-Poissonization, and Laplace convolutions. Our results
imply concentration of the profiles in the range where the mean tends
to infinity. Moreover, we also obtain a two-point concentration for
the distributions of the height and the saturation level.

\end{abstract}



\section{Introduction and Results}\label{intro}

Digital trees are fundamental data structures for words or alphabets
in computer algorithms whose analysis has attracted much attention
over the last half century. One major such variant is the
\emph{digital search tree} (DST for short), introduced by Coffman and
Eve in 1970 \cite{CoEv} (see also \cite{K} for more information).
Such structures are closely related to the popular Lempel-Ziv
compression scheme, and their asymptotic behaviors under random
inputs are often more challenging than those for other digital tree
families because of the natural occurrence of differential-functional
equations instead of purely algebraic-functional equations.

We begin with the definition of DST, which is the main object of
study in this paper. In the simplest situation, it is built from
digital data consisting of a sequence of records in the form of
$0$-$1$ strings. The first record is stored at the root of the tree.
All other records are distributed to the left- or right-subtree
according as their first bit being $0$ or $1$, respectively, and
retain their relative order. The subtrees of the root are then built
according to the same rules but by using the $j$th digit at level $j$
in further directing the strings to their respective subtrees. The
splitting process stops when the size of the subtree is either zero
or one. Note that the resulting tree is a binary tree with internal
nodes holding the records. External nodes, which represent places
where potential records can be inserted, are often added to the tree
(in fact, two external nodes are automatically created in the
algorithmic implementation for each new internal node); see Figure
\ref{dst-fig} for an example of a DST built from five records
(internal nodes are represented by rectangles and external nodes by
circles).

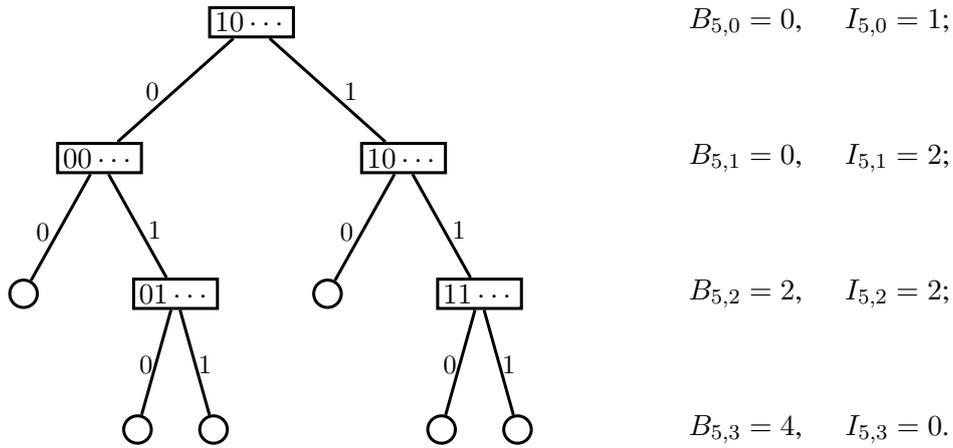
\begin{figure}[h]
\vspace*{0.4cm}
\begin{center}
\begin{tikzpicture}[line width= 0.4mm]
\draw (0cm,0cm) node[inner sep= 2pt,rectangle, draw] (1) {$10\cdots$};
\draw (-2cm,-1.8cm) node[inner sep= 2pt, rectangle, draw] (2)
{$00\cdots$};
\draw (2cm,-1.8cm) node[inner sep= 2pt, rectangle, draw] (3)
{$10\cdots$};
\draw (-3cm,-3.6cm) node[circle, draw] (4) {};
\draw (-1cm,-3.6cm) node[inner sep= 2pt, rectangle, draw] (5)
{$01\cdots$};
\draw (1cm,-3.6cm) node[circle, draw] (6) {};
\draw (3cm,-3.6cm) node[inner sep= 2pt, rectangle, draw] (7)
{$11\cdots$};
\draw (-1.5cm,-5.4cm) node[circle, draw] (8) {};
\draw (-0.5cm,-5.4cm) node[circle, draw] (9) {};
\draw (2.5cm,-5.4cm) node[circle, draw] (10) {};
\draw (3.5cm,-5.4cm) node[circle, draw] (11) {};

\draw (1)--(2) node[pos= 0.7,above] {\small{$0$}}; \draw (1)--(3)
node[pos= 0.7,above] {\small{$1$}};
\draw (2)--(4) node[pos= 0.8, above= 0.1cm] {\small{$0$}}; \draw
(2)--(5) node[pos= 0.8,above= 0.1cm] {\small{$1$}};
\draw (3)--(6) node[pos= 0.8,above= 0.1cm] {\small{$0$}}; \draw
(3)--(7) node[pos= 0.8,above= 0.1cm] {\small{$1$}};
\draw (5)--(8) node[pos= 0.85,above= 0.2cm] {\small{$0$}}; \draw
(5)--(9) node[pos= 0.85,above= 0.2cm] {\small{$1$}};
\draw (7)--(10) node[pos= 0.85,above= 0.2cm] {\small{$0$}}; \draw
(7)--(11) node[pos= 0.85,above= 0.2cm] {\small{$1$}};

\draw (6.5cm,0cm) node {$B_{5,0}= 0$,}; \draw (8.5cm,0cm) node
{$I_{5,0}= 1$;};
\draw (6.5cm,-1.8cm) node {$B_{5,1}= 0$,}; \draw (8.5cm,-1.8cm) node
{$I_{5,1}= 2$;};
\draw (6.5cm,-3.6cm) node {$B_{5,2}= 2$,}; \draw (8.5cm,-3.6cm) node
{$I_{5,2}= 2$;};
\draw (6.5cm,-5.4cm) node {$B_{5,3}= 4$,}; \draw (8.5cm,-5.4cm) node
{$I_{5,3}= 0$.};
\end{tikzpicture}
\end{center}
\caption{A DST built from $5$ records with its profiles.}
\label{dst-fig}
\end{figure}

For the purpose of analysis, we assume that bits in the input strings
are independent and identically distributed with a common
Bernoulli$(p)$ random variable with $0<p<1$. Throughout this paper,
we fix $p= \frac12$, namely, we consider only the symmetric case.
This random model is called the \emph{symmetric Bernoulli model} and
the corresponding random tree is referred to as a \emph{random
symmetric DST}. It represents a simple model with reasonably good
predictive power in general (for example, results holding in the
Bernoulli model often have similar forms in more general Markov
models; see \cite{LouSpaTang}).

Under such a random model, we study in this paper the external and
internal node profiles (referred to as the profiles for short) which
are defined as follows: the external profile of a random symmetric
DST of size $n$ is a double-indexed sequence of random variables
$B_{n,k}$ which counts the number of external nodes at (horizontal)
level $k$; the internal profile $I_{n,k}$ is similarly defined (with
external nodes replaced by internal nodes).

Profiles are \emph{fine shape characteristics} encoding the level
silhouette of the tree and they are closely connected to many other
shape parameters such as height, width, total path length, saturation
or fill-up level, and successful and unsuccessful search. In
particular, we will discuss the unsuccessful search (or the depth),
the height and the saturation level:

\begin{itemize}

\item Unsuccessful search $U_n$: the distance from the root to a
randomly chosen external node with its distribution given by
\begin{equation}\label{dis-un}
    \mathbb{P}(U_n= k)
	= \frac{\mathbb{E}(B_{n,k})}{n+1}.
\end{equation}

\item Height: the length of the longest path from the root to
an external node, or $\max\{k: B_{n,k}>0\}$;

\item Saturation level: the last level from the root which is
completely filled with internal nodes, or $\max\{k: I_{n,k}= 2^k\}$.

\end{itemize}
See for example \cite{Dr,HwFuZa} and the references therein for more
shape parameters in DSTs.

Historically, the external profile was among the very first shape
parameters analyzed on DSTs due to the connection \eqref{dis-un} to
the unsuccessful search; see Knuth \cite{K} and Konheim and Newman
\cite{KoNe}. Yet the understanding of the profiles of symmetric DSTs
has remained incomplete. Table~\ref{tab:3trees} summarizes the
current status for the profiles of tries, Patricia tries and DSTs,
the latter two representing other major classes of digital trees.

\def\arraystretch{1.3}
\begin{table}[!ht]
\begin{center}
\begin{tabular}{ccccc}
  Tree types & $p= q$? & Mean & Variance & CLT \\ \hline
  Tries & $0<p<1$ & \cite{PaHwNiSz} & \cite{PaHwNiSz} &
  \cite{PaHwNiSz} \\ \cline{2-5}
  \multirow{2}{*}{Patricia Tries} & $p= \frac12$
  & \cite{MaKnSz} & ? & ? \\
  & $p\ne\frac12$
  & \cite{DrMaSz,MaKnSz,MaSz} & \cite{DrMaSz,MaSz} & \cite{MaSz}
  \\ \cline{2-5}
  \multirow{2}{*}{DSTs} & $p= \frac12$
  & \cite{Lo,DrSz} & this paper & next paper \\
  & $p\ne\frac12$
  & \cite{DrSz}& \cite{KaVa} & ?\\ \hline
\end{tabular}
\end{center}
\vspace*{-.2cm}
\caption{\emph{A summary of the analysis in distribution of profiles
in the three major classes of digital trees under the Bernoulli
model.}}
\label{tab:3trees}
\end{table}

Briefly, in the case of random tries, the mean, the variance and the
asymptotic normality of both profiles under the symmetric and
asymmetric models are fully clarified in \cite{PaHwNiSz}. For
Patricia tries, the expected profiles were studied in \cite{MaKnSz}
for both symmetric and asymmetric models. Then the asymptotic
variance and the asymptotic normality of the profiles, \emph{inter
alia}, under the asymmetric model are established in the recent papers
\cite{DrMaSz,MaSz}.

As regards symmetric DSTs, Louchard \cite{Lo}, following
\cite{K,KoNe}, derived an explicit expression for the expected
profiles; see also \cite{Dr, DrSz,Ma,Pr}. Louchard also obtained an
asymptotic approximation for the mean profiles in the most important
range $k= \log_2n+\mathcal{O}(1)$ (where most nodes lie),
characterizing the asymptotic distribution of unsuccessful and
successful search. These results were later extended in \cite{Dr,
DrSz, KnSz, Ma}. We broaden the study in this paper to the variance
of the profiles for which an arduous analysis is carried out. From
our results we obtain concentration of the profiles around the expected profile in
the range where the mean becomes unbounded. This suggests that an
asymptotic normality result (in the sense of convergence in
distribution) holds in this range, too. We will deal with this
question in a companion paper due to very different techniques
required. Moreover, we will apply our results to the height and the
saturation level. See also \cite{De2,JacSpabook,LouSpaTang,Spabook}
for other parameters and different types of results on profiles in
DSTs.

We now state our results, focusing on the external profile. The
corresponding results for the internal profile will be given in
Section \ref{int-profile}. First, we introduce the following function
and sequence that are ubiquitous in the analysis of DSTs; see
\cite{K}.
\[
    Q(z)
    = \prod_{\ell\ge1}\lpa{1-2^{-\ell}z}
    \qquad\text{and}\qquad
    Q_n
    = \prod_{1\le \ell\le n}\lpa{1-2^{-\ell}}
    = \frac{Q(1)}{Q(2^{-n})}.
\]
Note that $\lim_{n\to\infty}Q_n$ exists and equals
$Q(1)=:Q_\infty$.

In the next section, we will derive the following (largely known) result
(see \cite{Dr}) for the mean of the external profile.
\begin{thm}\label{thm-mean}
The mean of the external profile satisfies
\begin{equation}\label{EXnk}
    \mathbb{E}(B_{n,k})
    = 2^kF\lpa{2^{-k}n}
	  +\mathcal{O}(1),
\end{equation}
uniformly for $0\le k\le n$, where $F(x)$ is a positive function on
$\mathbb{R}^+$ defined by
\begin{align}\label{F-sum}
    F(x)
    = \sum_{j\ge 0}\frac{(-1)^j2^{-\binom{j}{2}}}
    {Q_jQ_{\infty}}\,e^{-2^jx}.
\end{align}
\end{thm}
The proof of \eqref{EXnk} for small $k$, or more precisely, for $k$
such that $2^{-k}n\to\infty$, will follow readily by simple
elementary arguments, whereas that for the remaining range will rely
on complex-analytic tools. More precisely, when $2^{-k}n\to\infty$
and $k\ge1$, we will show that
\begin{equation}\label{asymp-mean-simp}
    \mathbb{E}(B_{n,k})
    = \frac{2^k}{Q_k}\lpa{1-2^{-k}}^n
    \lpa{1+\mathcal{O}
	\lpa{e^{-\frac{n}{2^k-1}}}},
\end{equation}
which is stronger than \eqref{EXnk} if $2^kF\lpa{2^{-k}n}
= \mathcal{O}(1)$.
\begin{Rem}
Note that \eqref{asymp-mean-simp} indeed holds for all $n$ and $k$
but is more useful in the range when $2^{-k}n\to\infty$.
\end{Rem}
\begin{figure}[ht]
\begin{center}
\includegraphics[height= 5cm,width= 6cm]{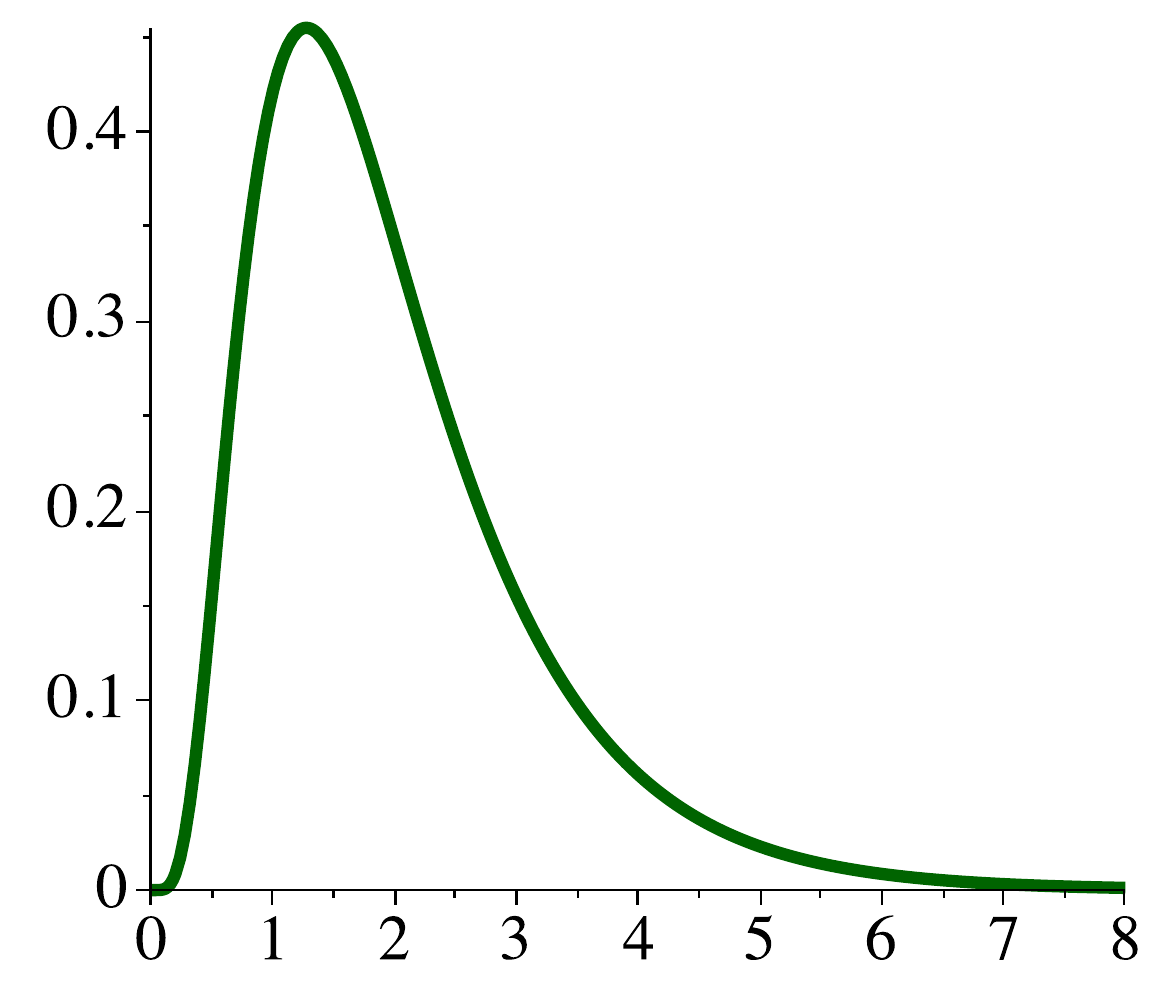}
\quad
\includegraphics[height= 5cm,width= 6cm]{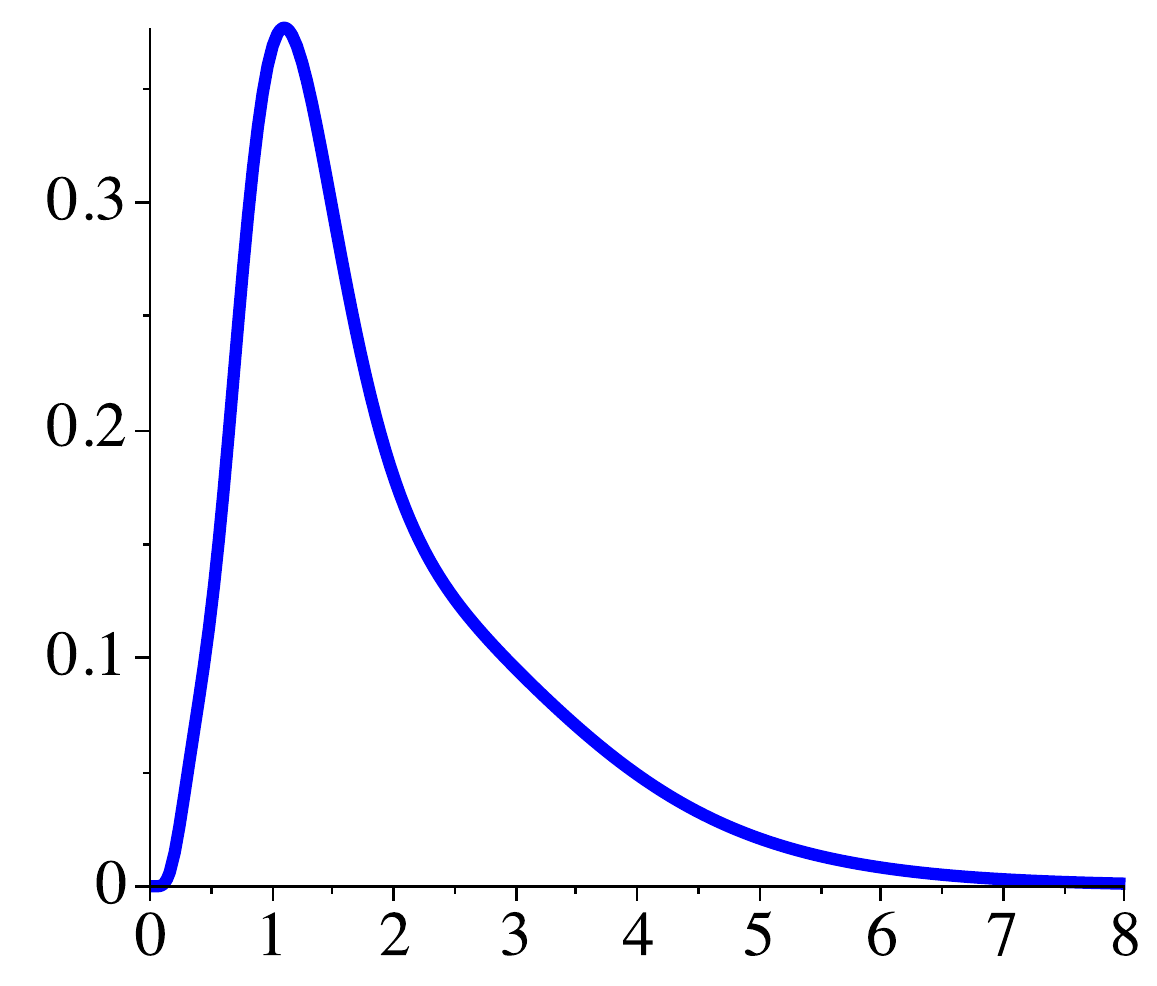}
\end{center}
\vspace*{-0.45cm}
\caption{\emph{The functions $F$ (left) and $G$ (right).}}
\label{plot-F}
\end{figure}

On the other hand, the relation \eqref{EXnk} is only a (useful)
asymptotic approximation if the first term on the right-hand side is
not bounded for large $n$. Thus, to understand when this holds, we
derive more precise asymptotic approximations of $F(x)$ for large and
small $x$; see Figure~\ref{plot-F} (left) for a graphical rendering
of $F$.

Observe first that the series definition \eqref{F-sum} of $F$ extends
to complex parameter $z$ with $\Re(z)\ge 0$ and is itself an
asymptotic expansion for large $|z|$:
\begin{equation}\label{Fx-large}
    F(z)
	= \frac{e^{-z}}{Q_{\infty}}
	    +\mathcal{O}\lpa{e^{-2\Re(z)}},
    \qquad(\Re(z)\ge 0).
\end{equation}
On the other hand, for small $x$ with $X := \frac1{x\log 2}$ (see
Proposition~\ref{Th5}),
\begin{align}\label{Fx-small}
    F(x)
    = \sqrt{\frac{\log 2}{2\pi}}(1+o(1))
	    X^{\frac12+\frac1{\log 2}}
	    \exp&\llpa{-\frac{(\log (X\log X))^2}{2\log 2}
	    -P(\log_2(X\log X))},
\end{align}
with $P(t)$, $t\in\mathbb{R}$, a $1$-periodic function whose Fourier
series is given explicitly by
\begin{align} \label{Pt}
    P(t)
	:= \frac{\log 2}{12}
		+\frac{\pi^2}{6\log 2}
	    -\sum_{j\ge1} \frac{\cos(2j\pi t)}
	    {j\sinh\lpa{\frac{2j\pi^2}{\log 2}}}.
\end{align}
Note that the series in \eqref{Pt}, representing the fluctuating part
of $P(t)$, $t\in\mathbb{R}$, has a (peak-to-peak) amplitude less than
$1.8\times 10^{-12}$. Both expansions \eqref{Fx-small} and \eqref{Pt}
can be extended to complex parameter; see Proposition~\ref{Th5}.

While it is well anticipated (from known results for tries and
Patricia tries) that $\mathrm{Var}(B_{n,k})$ is asymptotically of the
same form as \eqref{EXnk} for $\mathbb{E}(B_{n,k})$ in most ranges of
$k$ of interest, the function involved is surprisingly very
complicated, as shown in \eqref{G} below; see also Figure
\ref{plot-F} (right).

\begin{thm}\label{thm-var}
The variance of the external profile satisfies
\begin{align}\label{Var-Bnk}
    \mathrm{Var}(B_{n,k})
    = 2^kG\lpa{2^{-k}n}
		+\mathcal{O}(1),
\end{align}
uniformly for $0\le k\le n$, where $G(x)$ is a positive function on
$\mathbb{R}^+$ defined by
\begin{equation}\label{G}
	G(x)
	= \sum_{j,r\ge0}\sum_{0\le h,\ell\le j}
	\!\frac{(-1)^{r+h+\ell}
	2^{-j-\binom{r}{2}-\binom{h}{2}-\binom{\ell}{2}+2h+2\ell}}
	{Q_{\infty} Q_r Q_{h}Q_{j-h}
	Q_{\ell}Q_{j-\ell}}\,\varphi\lpa{2^{j+r},2^h+2^\ell;x},
\end{equation}
with
\begin{equation}\label{varphi}
	\varphi(u,v;x)
	 = e^{-vx}\int_{0}^{x}(x-t)e^{-(u-v)t}\dd{t}
	 = \begin{cases}
		\displaystyle
		e^{-vx}\frac{e^{-(u-v)x}-1+(u-v)x}{(u-v)^2},
		&\text{if}\ u\ne v;\\
		\frac12x^2e^{-ux},
		&\text{if}\ u= v.
	\end{cases}
\end{equation}
\end{thm}
\begin{Rem}
In the case when $2^{-k}n\to\infty$, we will in fact prove that
\[
    \mathbb{E}(B_{n,k})
	\sim \mathrm{Var}(B_{n,k}).
\]
\end{Rem}
Despite of its complicated form, the function $G$ is very close to $F$
in the following sense (see Section~\ref{var}):
\begin{equation}\label{asymp-rel-FG}
    G(x)\sim
	\begin{cases}
	    F(x), & \text{if } x\to\infty;\\
		2F(x),& \text{if }x\to 0;
	\end{cases}	
\end{equation}
see also \cite{PaHwNiSz} for the same type of results for symmetric
tries, and Devroye \cite{De} for a general bound for the profile
variance. If $x\to\infty$, then a more precise approximation is
\begin{equation}\label{Gx-large}
    G(x)
    = \frac{e^{-x}}{Q_{\infty}}
	    +\mathcal{O}\lpa{xe^{-2x}},
\end{equation}
where the second-order term differs from that of $F$; see
\eqref{Fx-large}.

The two theorems imply that the mean and the variance have very
similar behaviors. In particular, they tend to infinity in the same
range of $k$.
\begin{cor} \label{cor-iff}
For large $n$ and $0\le k\le n$, $\mathbb{E}(B_{n,k})\to\infty$ if
and only if $\,\mathrm{Var}(B_{n,k})\to\infty$. In fact,
$\mathbb{E}(B_{n,k})$ and $\,\mathrm{Var}(B_{n,k})$ are of the same
order of magnitude in this range.
\end{cor}

We now describe the range where the mean and the variance tend
to infinity. Define two functions of $n$:
\begin{equation}\label{kskh}
\begin{split}	
    k_s
	&:= \log_2n
		-\log_2\log n
		+\frac{\log_2\log n}
		{\log n},  \\
    k_h
	&:= \log_2n
		+ \sqrt{2\log_2n}
		-\frac12\log_2\log_2n
	    +1+\frac1{\log 2}
		-\frac{3\log\log n}
	    {4\sqrt{2(\log n)(\log 2)}}.
\end{split}
\end{equation}

\begin{cor}\label{cor-central-range}
The mean and the variance of $B_{n,k}$ tend to infinity if and only
if there exists a positive sequence $\omega_n$ tending to infinity
with $n$ such that
\begin{equation}\label{k0k1}
    k_s+\frac{\omega_n}{\log n}
	\le k
	\le k_h - \frac{\omega_n}{\sqrt{\log n}}.
\end{equation}
\end{cor}

The proof of this follows from a straightforward computation and is
left to the reader; see Section~\ref{app} for similar computations.
Note that the range is very small (or almost all nodes are
concentrated at these levels). For convenience, we refer to
\eqref{k0k1} as the \emph{central range}.

Since $\mathbb{E}(B_{n,k})$ and $\,\mathrm{Var}(B_{n,k})$ are of the
same order of magnitude as $\mathrm{Var}(B_{n,k})\to\infty$, we
immediately get the following property.
\begin{cor}\label{cor-cip}
If $\mathrm{Var}(B_{n,k})\to\infty$, then the external profile is
asymptotically concentrated around its expected value:
\[
	\frac{B_{n,k}}{\mathbb{E}(B_{n,k})}
	\stackrel{P}{\longrightarrow} 1,
\]
where $\stackrel{P}{\longrightarrow}$ denotes convergence in
probability.
\end{cor}

Corollary~\ref{cor-cip} suggests that the profile follows the central
limit theorem in the central range:
\begin{align}\label{E:B-clt}
	\frac{B_{n,k}-\mathbb{E}(B_{n,k})}
	{\sqrt{\mathrm{Var}(B_{n,k})}}
	\stackrel{d}{\longrightarrow} \mathscr{N}(0,1),
\end{align}
where $\mathscr{N}(0,1)$ denotes the standard normal distribution.
Indeed, we can prove \eqref{E:B-clt} when the variance tends to infinity not too slowly
by using the methods of this paper and the contraction method;
compare with Neininger and R\"{u}schendorf \cite{NeRu04}. However, in
order to obtain \eqref{E:B-clt} for the whole central range, a finer
analysis of the asymptotic behaviors of $\mathbb{E}(B_{n,k})$ and
$\mathrm{Var}(B_{n,k})$ is required, as well as a delicate bivariate
asymptotic transfer result for recurrences of the type
(\ref{rec-mom}). These will be addressed in a companion paper.

Results of a very similar nature for the internal profile are given in
Section~\ref{int-profile}.

These new results for the internal and external profile have many
consequences in view of their close connections to other shape
parameters. We content ourselves here with stating an application to
the height $H_n$ of DSTs, which is related to $B_{n,k}$ by $H_n :=
\max\{k: B_{n,k}>0\}$; see Section~\ref{app} for other consequences.

\begin{thm}\label{Th:height}
Define $k_H$ as
\begin{equation}\label{kH}
    k_H
	= \left\lfloor \log_2n
		+\sqrt{2\log_2n}
		-\frac{1}{2}\log_2\log_2n
		+\frac{1}{\log 2}
	\right\rfloor,
\end{equation}
which is at the upper boundary of the central range \eqref{k0k1}.
Then the distribution of $H_n$ is concentrated at the two points
$k_H$ and $k_H+1$:
\begin{align}\label{Hn-2points}
    \mathbb{P}(H_n= k_H \text{ or } H_n= k_H+1)
	\to 1, \qquad (n\to\infty).
\end{align}
\end{thm}

The possibility that such a result might hold was mentioned in
\cite{AlSh} for a closely related model; a heuristic derivation was
given in \cite{KnSz2}. See also \cite{And,Andetal} for other
two-point approximation results in probability theory.

It is interesting to compare \eqref{Hn-2points} with known results
for the height of tries and those for Patricia tries, which we
summarize in Table~\ref{tab-height-1}; see Flajolet \cite{Fla} for
the height of symmetric tries, and Knessl and Szpankowski \cite{KS02}
for that of Patricia tries (with only non-rigorous proofs).

\vspace*{-0.3cm}
\begin{table}[!ht]
\begin{center}
\begin{tabular}{clcc}
\multicolumn{4}{c}{{}} \\
\multicolumn{1}{c}{Tree types} &
\multicolumn{1}{c}{Expected height} &
\multicolumn{1}{c}{$\begin{array}{c}
\text{Discrete} \\
\text{concentration}
\end{array}$} &
\multicolumn{1}{c}{References}\\ \hline
Tries & $2\log_2n+\mathcal{O}(1)$ & no & \cite{Fla} \\
Patricia tries
& $\log_2n +\sqrt{2\log_2n}+\mathcal{O}(1)$ & at 3 pts
& \cite{KS02} (non-rigorous)\\
DSTs & $\!\!\!\begin{array}{l}
\log_2n +\sqrt{2\log_2n} \\
\quad -\frac12\log_2\log_2n +\mathcal{O}(1)
\end{array}$ & at 2 pts & this paper\\ \hline
\end{tabular}	
\end{center}
\vspace*{-.3cm}
\caption{\emph{A comparison of the height of random symmetric
tries, Patricia tries and DSTs.}} \label{tab-height-1}
\end{table}

\vspace*{0.3cm}
We describe briefly the methods and tools used in proving
Theorems~\ref{thm-mean}--\ref{thm-var}, which both start with the
following distributional recurrence
\begin{equation}\label{dist-rec-exprof}
    B_{n,k}
    \stackrel{d}{= }B_{J_n,k-1}
	    +B_{n-1-J_n,k-1}^{*},
    \qquad (n, k\ge 1),
\end{equation}
with the boundary conditions $B_{0,0}= 1, B_{0,k}= 0$ for $k\ge 1$,
$B_{n,0}= 0$ for $n\ge 1$, where $J_n = \text{Binomial}
\lpa{n-1,\tfrac12}$, and $B_{n,k}^{*}$ is an independent copy of
$B_{n,k}$.

To derive the asymptotic approximations of the mean (Theorem
\ref{thm-mean}) and the variance (Theorem \ref{thm-var}), we rely on
the property, in view of \eqref{dist-rec-exprof}, that all moments of
$B_{n,k}$ satisfy recurrences of the following type:
\begin{equation}\label{rec-mom}
    a_{n,k}
    = 2^{2-n}\sum_{0\le j<n}\binom{n-1}{j}a_{j,k-1}
	    +b_{n,k},
\end{equation}
for some given sequence $b_{n,k}$. This recurrence looks standard but
the complication here comes from the dependence of $k$ on $n$. When
$k$ is small, more precisely, when $2^{-k}n\to\infty$, the tree shape
at these levels has little variation and thus both mean and variance
can be treated by simple elementary arguments. The hard ranges are
when $2^{-k}n\asymp 1$ and $2^{-k}n\to0$ for which our arguments are
built upon the idea of \emph{Poissonization} by defining the
\emph{Poisson generating functions}
\[
    \tilde{A}_k(z)
    = e^{-z}\sum_{n\ge0}a_{n,k}\frac{z^n}{n!}
	\quad\text{and}\quad
    \tilde{B}_k(z)
    = e^{-z}\sum_{n\ge0}b_{n,k}\frac{z^n}{n!}.
\]
Then \eqref{rec-mom} is translated into the differential-functional
equation
\[
    \tilde{A}_k(z)+\tilde{A}'_k(z)
    = 2\tilde{A}_{k-1}\lpa{\tfrac12z}
	    +\tilde{B}_k(z),
\]
which amounts to describing the moments in the \emph{Poisson model}.
This equation is solved via Laplace transform techniques, which lead
to exact and asymptotic expressions whose asymptotic properties are
further examined via Mellin transform, saddle-point method and again
Laplace transform. Finally, we translate the results in the Poisson
model to those in the Bernoulli model via de-Poissonization.

While these procedures are by now standard (see
\cite{FuHwZa,HwFuZa}) and work well for the mean, the analysis
of the variance is more subtle. Here, the most crucial step is to
introduce a \emph{Poissonized variance} in the Poisson model
(see again \cite{FuHwZa,HwFuZa}) so as to provide an asymptotic
equivalent to the variance after de-Poissonization. An appropriate
adaptation in the current situation is to define the function
\[
    \tilde{V}_k(z)
    := \tilde{M}_{k,2}(z)
		-\tilde{M}_{k,1}(z)^2
	    -z\tilde{M}'_{k,1}(z)^2,
\]
where $\tilde{M}_{k,2}(z)$ and $\tilde{M}_{k,1}(z)$ denote the
Poisson generating functions of the second moment and the first
moment of $B_{n,k}$, respectively. Then we show that in the central range
$\tilde{V}_k(z)$ is well-approximated by $2^kG\lpa{2^{-k}z}$ for large $|z|$, and that
$\tilde{V}_k(n)$ is asymptotically equivalent to
$\mathrm{Var}(B_{n,k})$. Once these are clarified, the next challenge
is the asymptotic behavior of $G(z)$, notably for small $|z|$, which
turns out to be the most technical part of this paper (see
Proposition~\ref{Th7}), largely due to the complicated form of the
Laplace transform of $G(z)$ (see \eqref{lap-H} and \eqref{lap-G-1})
and the uniformity required for large parameters (see
Lemma~\ref{LeG2}).

Finally, Theorem \ref{Th:height} and some related properties will be
proved in Section \ref{app} by the first and second moment method,
relying on the estimates from Sections \ref{mean}--\ref{int-profile}.
The corresponding asymptotic properties for the internal profiles will
be given in Section~\ref{int-profile}.

An extended abstract of this paper (entitled \emph{External Profile of
Symmetric Digital Search Trees}) by the same authors has appeared in
the \emph{Proceedings of the Fourteenth Workshop on Analytic
Algorithmics and Combinatorics (ANALCO17)}, and contains Theorems
\ref{thm-mean}--\ref{Th:height} and sketches of the proofs of the
first two. With the exception of the limiting distributions, the
current paper provides the proofs and derives additionally the same
types of asymptotic approximations for the internal profile and
discusses some of their consequences.

\section{Expected Values of the External Profile}\label{mean}

In this section, we prove Theorem \ref{thm-mean} for
$\mathbb{E}(B_{n,k})$. As mentioned in the Introduction, most results
given here are known. Nevertheless, we provide detailed proofs
because the analysis of the variance will follow the same pattern.

We start from \eqref{dist-rec-exprof}. Write $\mu_{n,k}
= \mathbb{E}(B_{n,k})$. Then
\[
    \mu_{n,k}
    = 2^{2-n}\sum_{0\le j<n}
    \binom{n-1}{j}\mu_{j,k-1},
    \qquad (n,k\ge 1),
\]
with the boundary conditions $\mu_{0,0}= 1$, and $\mu_{n,0} =
\mu_{0,k}= 0$ for $n,k\ge1$. We then consider the Poisson generating
function
\[
    \tilde{M}_{k,1}(z)
    := e^{-z}\sum_{n\ge 0}\mu_{n,k}\frac{z^n}{n!},
    \qquad(k\ge0),
\]
which satisfies the differential-functional equation
\begin{equation}\label{func-mean}
    \tilde{M}_{k,1}(z)+\tilde{M}'_{k,1}(z)
    = 2\tilde{M}_{k-1,1}\lpa{\tfrac12z},
    \qquad (k\ge 1),
\end{equation}
with $\tilde{M}_{0,1}(z)= e^{-z}$.

We now solve this differential-functional equation using Laplace
transform, which, by inverting and taking coefficients,
leads to an exact expression for $\mu_{n,k}$.

\subsection{Exact Expressions}

To solve \eqref{func-mean}, we apply the Laplace transform
(subsequently denoted by $\mathscr{L}[\cdot;s]$) on both sides of
\eqref{func-mean}, and obtain
\[
    \mathscr{L}[\tilde{M}_{k,1}(z);s]
    = \frac{4}{s+1}\,\mathscr{L}[\tilde{M}_{k-1,1}(z);2s],
    \qquad (k\ge 1),
\]
with $\mathscr{L}[\tilde{M}_{0,1}(z);s] = \frac1{s+1}$. A direct
iteration then yields
\begin{equation}\label{lap-Mk1-prod}
    \mathscr{L}[\tilde{M}_{k,1}(z);s]
    = \frac{4^k}{(s+1)(2s+1)\cdots(2^ks+1)},
\end{equation}
for $k\ge0$. By partial fraction expansion, we see that
\begin{equation}\label{lap-Mk1}
    \mathscr{L}[\tilde{M}_{k,1}(z);s]
    = 2^k\sum_{0\le j\le k}
    \frac{(-1)^j2^{-\binom{j}{2}}}{Q_jQ_{k-j}}
    \cdot\frac{1}{s+2^{j-k}},
\end{equation}
which, by term-by-term inversion, gives
\begin{equation}\label{poi-mean}
    \tilde{M}_{k,1}(z)
    = 2^k\sum_{0\le j\le k}
	\frac{(-1)^j2^{-\binom{j}{2}}}
    {Q_jQ_{k-j}}\,e^{-2^{j-k}z}, \qquad(k\ge0).
\end{equation}
From this, we obtain the closed-form expression for the expected
profile (first derived in \cite{Lo})
\begin{equation}\label{bernoulli-mean}
    \mu_{n,k}
    = 2^k\sum_{0\le j\le k}
	\frac{(-1)^j2^{-\binom{j}{2}}}
    {Q_jQ_{k-j}}\,\lpa{1-2^{j-k}}^n.
\end{equation}
We now examine the asymptotic aspects.

\subsection{Asymptotics of $\mu_{n,k}$}

If $2^{-k}n\to\infty$, then an expansion for the mean can be derived
by elementary arguments because the term in \eqref{bernoulli-mean}
with $j= 0$ is dominant. More precisely, we have
\begin{align}
    \mu_{n,k}
    &= \frac{2^k}{Q_k}\lpa{1-2^{-k}}^n
    \llpa{1+\mathcal{O}\llpa{
    \frac{\lpa{1-2^{1-k}}^n}{\lpa{1-2^{-k}}^n}}},
    \label{munk-1}
\end{align}
where the error term is bounded above by
\begin{align} \label{two-k-ratio}
    \frac{\lpa{1-2^{1-k}}^n}{\lpa{1-2^{-k}}^n}
    = \exp\llpa{-n\sum_{\ell\ge 1}
    \frac{2^{\ell}-1}{\ell 2^{k\ell}}}
    \le \exp\llpa{-\frac{n}{2^k-1}}
	\qquad(k\ge1).
\end{align}
Substituting this into \eqref{munk-1} proves the asymptotic estimate
\eqref{asymp-mean-simp} for $\mu_{n,k}$ when $2^{-k}n\to\infty$ and
$k\ge 0$.

If $2^{-k}n= \mathcal{O}(1)$, then no single term in
\eqref{bernoulli-mean} is dominating and $2^{-rk}n\to 0$ for $r\ge2$,
so we readily obtain, again by \eqref{bernoulli-mean} (approximating
$(1-x)^n$ by $e^{-xn}$ and by extending $k$ to infinity), $\mu_{n,k}
\sim 2^kF\lpa{2^{-k}n}$, but the asymptotics of $F$ for small
parameter remains unclear. We use instead the Poissonization
techniques (see \cite{HwFuZa,JaSz}) to derive the required asymptotic
approximation of $\mu_{n,k}$; see Theorem \ref{thm-mean}.

We derive first a simple bound for $F(z)$ and its derivatives.
\begin{lmm}\label{bounds-F}
For $m\ge0$ and $\Re(z)\ge0$, the $m$th derivative of $F$ satisfies
the uniform bound
\begin{align}\label{Frz-ub}
    \sup_{\Re(z)\ge0}\bigl| F^{(m)}(z)\bigr|
    = \mathcal{O}\lpa{2^{\binom{m+1}{2}}},
\end{align}
where the implied constant is absolute.
\end{lmm}
\begin{proof} By the definition \eqref{F-sum}:
\[
    |F^{(m)}(z)|
    \le\sum_{j\ge 0}\frac{2^{-\binom{j}{2}+jm}}
	{Q_jQ_{\infty}}\,e^{-2^j\Re(z)}
    \le Q_\infty^{-2}e^{-\Re(z)}
    \sum_{j\ge 0}2^{-\binom{j}{2}+jm}.
\]
Now	
\[
    \sum_{j\ge 0}2^{-\binom{j}{2}+jm}
	\le 2^{\frac12(m+\frac12)^2}\sum_{j\in\mathbb{Z}}
	2^{-\frac12(j-m-\frac12)^2}
    = 2^{\frac12(m+\frac12)^2+1}\sum_{j\ge1}
	2^{-\frac12(j-\frac12)^2}.
\]
This proves the uniform bound \eqref{Frz-ub}. \end{proof}

We then show that \eqref{poi-mean} can be brought into the following
more useful form (both exact and asymptotic).
\begin{lmm}\label{id-poi-mean}
For $\Re(z)\ge0$ and $k\ge 1$, the Poisson generating function
$\tilde{M}_{k,1}(z)$ of the expected profile $\mu_{n,k}$ satisfies
\[
    \tilde{M}_{k,1}(z)
    = 2^k\sum_{m\ge 0}
		\frac{2^{-\binom{m+1}{2}-km}}{Q_m}\,
	    F^{(m)}\lpa{2^{-k}z},
\]
where $F(z)$ is given in \eqref{F-sum}.
\end{lmm}
\begin{proof} By Euler's identity (see \cite[Corollary 2.2]{An})
\begin{align*}
    \sum_{j\ge 0}
	    \frac{(-1)^j q^{\binom{j}{2}}}
	    {(1-q)(1-q^2)\cdots(1-q^j)}z^j
    = \prod_{\ell\ge0}
	     \lpa{1-q^\ell z},
    \qquad(0<q<1),
\end{align*}
we have
\begin{align*}
    \frac {Q_\infty}{Q_{k-j}}
    = \prod_{\ell\ge1}
	    \lpa{1 - 2^{j-k-\ell}}
    = \sum_{m\ge 0}
	    \frac{(-1)^m 2^{-{m+1\choose 2}}}{Q_m}\,2^{(j-k)m},
\end{align*}
which is still valid for $j>k$ (in which case both sides are zero).
Substituting the latter into \eqref{poi-mean} gives
\begin{align*}
    \tilde{M}_{k,1}(z)
	&= 2^k \sum_{0\le j\le k}
    \frac{(-1)^{j} 2^{-\binom{j}{2}}}
    {Q_j Q_{k-j}}\,e^{-2^{j-k}z} \\
    &= 2^k \sum_{j\ge0}
	\frac{(-1)^{j} 2^{-\binom{j}{2}}}{Q_jQ_\infty}
	\sum_{m\ge 0}
	\frac{(-1)^m 2^{-\binom{m+1}{2}+(j-k)m}}
	{Q_m}\,e^{-2^{j-k}z} \\
    &= 2^k \sum_{m\ge 0}
	\frac{(-1)^m 2^{-\binom{m+1}{2}-km}}{Q_m Q_\infty}
	\sum_{j\ge0}
	\frac{(-1)^{j} 2^{-\binom{j}{2}+jm}}
	{Q_j}\,e^{-2^{j-k} z},
\end{align*}
where interchanging the sums is justified as in the proof of
Lemma~\ref{bounds-F}. This proves the lemma since the last series is
equal to $(-1)^mQ_{\infty}F^{(m)}\lpa{2^{-k}z}$.
\end{proof}

From these two lemmas, we get
\begin{equation}\label{exp-Mk1}
    \tilde{M}_{k,1}(z)
    = 2^kF\lpa{2^{-k}z}
	    +\mathcal{O}(1),
    \qquad (\Re(z)\ge 0),
\end{equation}
which is the Poissonized version of \eqref{EXnk}.

The asymptotics of $\mu_{n,k}$ and that of $\tilde{M}_k(n)$ can be
bridged by the analytic de-Poissonization techniques; see the survey
paper \cite{JaSz}. For that purpose, it turns out that the use of
JS-admissible functions, a notion introduced in \cite{HwFuZa},
provides more operational flexibility.

Throughout this paper, the generic symbols $\ve, \ve'$ always denote
small positive quantities whose values are immaterial and not
necessarily the same at each occurrence.

\begin{Def}[\cite{HwFuZa}] An entire function $\tilde{f}(z)$ is said
to be JS-admissible, denoted by $\tilde{f}(z)\in\JS$, if the
following two conditions hold for $|z|\ge1$.
\begin{itemize}

\item[\textbf{(I)}] There exists a constant $\alpha\in\mathbb{R}$ such
that $\tilde{f}(z) = \mathcal{O}\lpa{|z|^{\alpha}}$ uniformly for
$|\arg(z)|\le \ve$.

\item[\textbf{(O)}] Uniformly for $\ve\le |\arg(z)|\le \pi$,
$
    f(z)
    := e^{z}\tilde{f}(z)
    = \mathcal{O}\lpa{e^{(1-\ve')|z|}}.
$
\end{itemize}
\end{Def}

When $\tilde{f}\in\JS$, its coefficients can be expressed in terms of
the Poisson-Charlier expansion (see \cite{HwFuZa})
\begin{align}\label{pc}
    n![z^n]e^z\tilde{f}(z)
    = \sum_{j\ge0}
	\frac{\tilde{f}^{(j)}(n)}{j!}\, \tau_j(n),
\end{align}
which is not only an identity but also an asymptotic expansion, where
the $\tau_j(n)$'s are essentially Charlier polynomials defined by
\[
    \tau_j(n)
	= n![t^n]e^t(t-n)^j
    = \sum_{0\le \ell\le j}
		\binom{j}{\ell}(-n)^{j-\ell}\frac{n!}{(n-\ell)!},
	    \qquad(j= 0,1,\dots).
\]
In particular, $\tau_j(n)$ is a polynomial in $n$ of degree
$\lfloor\frac12j\rfloor$; the expressions for $\tau_j(n)$, $0\le j\le
5$, are given below.
\begin{center}
\begin{tabular}{cccccc}
    $\tau_0(n)$ & $\tau_1(n)$ & $\tau_2(n)$ & $\tau_3(n)$ &
    $\tau_4(n)$ & $\tau_5(n)$ \\ \hline
    $1$ & $0$ & $-n$ & $2n$ & $3n(n-2)$ &
    $-4n(5n-6)$
\end{tabular}
\end{center}

For our purpose, we also need an additional uniformity property
for JS-admissible functions as the level parameter $k$ may also
depend on $n$.
\begin{lmm}\label{JS-Mk1}
The functions $\tilde{M}_{k,1}(z)$ are uniformly JS-admissible:
if $\vert z\vert\ge 1$, then for $\vert\arg(z)\vert\le \ve$
\begin{align*}
    \tilde{M}_{k,1}(z)
    = \mathcal{O}(\vert z\vert)
\end{align*}
and for $\ve\le \vert\arg(z)\vert\le \pi$
\begin{align}\label{Mk1-o}
    e^z\tilde{M}_{k,1}(z)
    = \mathcal{O}\lpa{e^{(1-\ve')\vert z\vert}},
\end{align}
where all implied constants are absolute and hold uniformly for
$k\ge 0$.
\end{lmm}
\begin{proof} Consider first the region $\vert\arg(z)\vert\ge\ve$.
Let $M_{k,1}(z):= e^{z}\tilde{M}_{k,1}(z)$. The bound \eqref{Mk1-o}
holds trivially for $k= 0$ since $M_{0,1}(z)= 1$. Thus, we assume
$k\ge 1$. Then, we can rewrite \eqref{func-mean} into the following
form
\begin{equation}\label{equ-Mk1}
    M_{k,1}(z)
    = \int_{0}^{z}
	    2e^{\frac12u}M_{k-1,1}\lpa{\tfrac12u}\dd{u}
    = 2z\int_{0}^{1}
	    e^{\frac12tz}M_{k-1,1}\lpa{\tfrac12tz}\dd{t}.
\end{equation}
By the trivial bound $\mu_{n,k}\le 2n$, we get the \emph{a priori}
upper estimate $\vert M_{k,1}(z)\vert\le 2\vert z\vert e^{\vert
z\vert}$ (which also holds for $k= 0$). Plugging this into
\eqref{equ-Mk1} yields
\begin{align*}
    \vert M_{k,1}(z)\vert
	&\le 2\vert z\vert^2\int_{0}^{1}
	    te^{\frac12t(\Re(z)+\vert z\vert)}\dd{t}\\
    &\le 2\vert z\vert^2\int_{0}^{1}
        e^{\frac12t(\cos\ve+1)\vert z\vert}\dd{t}\\
    &\le \frac{4\vert z\vert}{\cos\ve+1}\,
        e^{\frac12(\cos\ve+1)\vert z\vert}.
\end{align*}
Since $\frac12<\frac12(\cos\ve+1)<1$, this proves \eqref{Mk1-o}.

Now we consider the sector $\vert\arg(z)\vert\le \ve$. The required
bound $\tilde{M}_{k,1}(z)= \mathcal{O}(|z|)$ will follow from
\eqref{exp-Mk1} and the smallness of $F(z)/z$ to be proved in
Proposition \ref{Th5} below (see also Remark~\ref{F-for-small-z}). It
is also possible to give a direct proof, although with a weaker
estimate (sufficient for our de-Poissonization purposes, however,
which would lead to weaker error terms in the subsequent estimates).
By
\eqref{equ-Mk1}
\[
    \tilde{M}_{k,1}(z)
    = 2z\int_{0}^{1}
	    e^{-(1-t)z}\tilde{M}_{k-1,1}\lpa{\tfrac12tz}\dd{t},
\]
and induction on $k$, we deduce the (slightly worse) bound
$\tilde{M}_{k,1}(z) = \mathcal{O}(\vert z\vert^{1+\ve'})$.
\end{proof}

By the asymptotic expansion \eqref{pc} and Lemma~\ref{JS-Mk1} (which
also gives bounds on the derivatives of $\tilde{M}_{k,1}(z)$ by
Ritt's theorem; see \cite{HwFuZa}), we can justify the ``Poisson
heuristic" for $\mu_{n,k}$ as follows.
\begin{lmm} For large $n$ and $0\le k\le n$
\begin{align*}
    \mu_{n,k}
    = \tilde{M}_{k,1}(n)
	    +\mathcal{O}(1),
\end{align*}
where the $\mathcal{O}$-term holds uniformly in $k$.
\end{lmm}

From this and \eqref{exp-Mk1}, we obtain now Theorem~\ref{thm-mean}
(except for the positivity of $F(x)$, which will be established
in Proposition~\ref{F-pos} below).

\subsection{Asymptotics of $F(z)$}\label{asymp-F}

In this subsection, we derive an asymptotic expansion for $F(z)$ for
small $|z|$ and prove the positivity of $F(x)$ on $\mathbb{R}^+$; the
corresponding large-$|z|$ asymptotics is much easier; see
\eqref{Fx-large}.

\begin{prop}\label{Th5}
For each integer $m\ge 0$, the $m$th derivative of $F$ satisfies
\begin{equation}\label{saddle-point-Fr}
    F^{(m)}(z)
    = \frac{\rho^{m+\frac12+\frac1{\log 2}}}
    {\sqrt{2\pi\log_2\rho}}\,
    \exp\Lpa{-\frac{(\log\rho)^2}{2\log 2}
    -P(\log_2\rho)}
    \lpa{1+\mathcal{O}\lpa{|\log\rho|^{-1}}},
\end{equation}
as $|z|\to0$ in the sector $|\arg(z)|\le \ve$, where $P(t)$ is given
in \eqref{Pt} and $\rho = \rho(z)$ solves the saddle-point equation
$
    \frac{\rho}{\log\rho}
	= \frac{1}{z\log 2},
$
and satisfies $|\rho|\to\infty$ as $|z|\to0$.
\end{prop}
\begin{proof}
By additivity, the Laplace transform of $F$ has the form
\[
    \mathscr{L}[F(z);s]
    = \sum_{j\ge 0}\frac{(-1)^j2^{-\binom{j}{2}}}
    {Q_jQ_{\infty}(s+2^j)}, \qquad (\Re(s)>-1),
\]
which equals the partial fraction expansion of the product
\begin{equation}\label{laplace-F}
    \mathscr{L}[F(z);s]
    = \prod_{j\ge 0}\frac{1}{1+2^{-j}s}
    = \frac{1}{Q(-2s)}.
\end{equation}

Since we are interested in the asymptotics of $F(z)$ as $|z|\to 0$,
which is reflected by the large-$|s|$ asymptotics of
$\mathscr{L}[F(z);s]$, we apply Mellin transform techniques for
that purpose; see Flajolet et al.'s survey paper \cite{FlGoDu} for
more background on tools and applications. In particular, taking the
logarithm on both sides of \eqref{laplace-F} and using the inverse
Mellin transform gives
\begin{align*}
    \log Q(-2s)
    = \sum_{j\ge 0}\log(1+2^{-j}s)
    = \frac{1}{2\pi i}
	\int_{-\tfrac12-i\infty}^{-\tfrac12+i\infty}
	    \frac{\pi s^{-\omega}}{(1-2^{\omega})
	    \omega\sin\pi\omega}\dd\omega,
\end{align*}
because the Mellin transform of $\log(1+s)$ equals
\[
    \int_{0}^{\infty}s^{\omega-1}\log(1+s)\dd{s}
    = \frac{\pi}{\omega\sin\pi\omega},
    \qquad (\Re(\omega)\in(-1,0)).
\]
Now, by standard Mellin analysis (see \cite{FlGoDu}), we deduce that
\begin{align}\label{Q2s-asymp}
    \log Q(-2s)
    = \frac{(\log s)^2}{2\log 2}+\frac{\log s}{2}
    +P(\log_2 s)+\mathcal{O}\lpa{\vert s\vert^{-1}},
\end{align}
uniformly as $\vert s\vert\to\infty$ in the sector $\vert\arg
s\vert\le \pi-\ve$.

Next, by the inverse Laplace transform, first for $z=x$ real, we have
\begin{equation}\label{inv-Lap-Fm}
    F^{(m)}(x)
    = \frac{1}{2\pi i}\int_{1-i\infty}^{1+i\infty}
    \frac{s^me^{xs}}{Q(-2s)}\dd{s}.
\end{equation}
It follows, by moving the line of integration to $\Re(s)= \rho$ and by
substituting the asymptotic approximation \eqref{Q2s-asymp}, that
\begin{align*}
    F^{(m)}(x)
    &= \frac1{2\pi i}\int_{\rho-i\infty}^{\rho+i\infty}
    s^m \exp\llpa{xs-\frac{(\log s)^2}{2\log 2}-\frac{\log s}{2}
    -P(\log_2 s)+\mathcal{O}\lpa{\vert s\vert^{-1}}}
    \dd s.
\end{align*}
A standard application of the saddle-point method (see \cite[Ch.\
VIII]{FlSe}) then yields \eqref{saddle-point-Fr} for real $z$ with
$z\to 0$.

When the imaginary part of $z$ is not zero, we can still apply the
same procedure but need to deform the integration contour in the
representation \eqref{inv-Lap-Fm} from the vertical line with real
part $1$ to the one where the two portions from $1+i$ to $1+i\infty$
and from $1-i\infty$ to $1-i$ are tilted slightly to the left; see
Appendix~\ref{A:A} for details.
\end{proof}

\begin{Rem}\label{F-for-small-z}
As a consequence of the above proposition, we see that $F(z)$ is
smaller than any polynomial of $z$ as $|z|\to 0$ in the sector
$\vert\arg(z)\vert\le \ve$.
\end{Rem}

Asymptotically, for large $X := \frac1{x\log 2}$, $x\in\mathbb{R}$,
\begin{align}
    \rho
    = X\Bigg(\log X+\log\log X+\frac{\log\log X}{\log X}
    -\frac{(\log\log X)^2-2\log\log X}
    {2(\log X)^2}+\mathcal{O}
    \llpa{\frac{(\log\log X)^3}{(\log X)^3}}\Bigg).
\end{align}
Substituting this into \eqref{saddle-point-Fr} gives the more explicit
expression \eqref{Fx-small}.

Finally, we prove the positivity of $F$ on the positive real line.
\begin{prop}\label{F-pos}
The function $F(x)$ is positive in $(0,\infty)$.
\end{prop}
\begin{proof}
Since $\mu_{n,k}\ge 0$, we see, by Theorem \ref{thm-mean}, that
$F(x)\ge 0$ on $(0,\infty)$. Then, from \eqref{laplace-F},
\[
    (1+s)\mathscr{L}[F(x);s]
    = \prod_{j\ge 0}\frac{1}{1+2^{-j-1}s}
    = \mathscr{L}[F(x);\tfrac12s].
\]
The corresponding inverse Laplace transform yields the equation
$F(x)+F'(x)= 2F(2x)$. With this differential-functional equation, we
prove the positivity of $F$ by contradiction. Assume \emph{a
contrario} that $F(x)$ has a zero in $(0,\infty)$, say $x_0$. Then
$F'(x_0)= 2F(2x_0)\ge 0$ and since $F'(x_0)>0$ is not possible (for
otherwise $F(x)$ would become negative in a neighborhood of $x_0$),
we also have $F(2x_0)= 0$. Continuing this argument, we obtain
arbitrarily large zeros. This is, however, impossible since we see
from \eqref{Fx-large} that $F(x)$ is positive for all $x$ large
enough. \end{proof}

\section{The Variance of the External Profile}\label{var}

In this section, we prove Theorem \ref{thm-var} by the same approach
used above for the mean, starting from the second moment
$\nu_{n,k}:= \mathbb{E}(B_{n,k}^2)$, which satisfies, by
\eqref{dist-rec-exprof}, the recurrence
\[
    \nu_{n,k}
    = 2^{2-n}\sum_{0\le j<n}\binom{n-1}{j}\nu_{j,k-1}
    +2^{2-n}\sum_{0\le j<n}\binom{n-1}{j}\mu_{j,k-1}\mu_{n-1-j,k-1},
    \qquad (n,k\ge 1),
\]
with the boundary conditions $\nu_{0,0}= 1$, and $\nu_{n,0}=
\nu_{0,k}=0$ for $n,k\ge1$. Translating this recurrence into the
corresponding Poisson generating function
\[
    \tilde{M}_{k,2}(z)
    := e^{-z}\sum_{n\ge 0}\nu_{n,k}\frac{z^n}{n!},
	\qquad(k\ge0),
\]
leads to the differential-functional equation
\begin{equation}\label{diff-func-sm}
    \tilde{M}_{k,2}(z)+\tilde{M}'_{k,2}(z)
    = 2\tilde{M}_{k-1,2}\lpa{\tfrac12z}
    +2\tilde{M}_{k-1,1}\lpa{\tfrac12z}^2,\qquad (k\ge 1),
\end{equation}
with $\tilde{M}_{0,2}(z)= e^{-z}$.

Since the variance is expected to be of the same order as the mean
(notably when both tend to infinity), there is a cancellation between
the dominant term in the asymptotic expansion for $\nu_{n,k}$ and
that for $\mu_{n,k}^2$. Such a cancellation of dominant terms can be
more easily manipulated by considering the Poissonized variance (as
in \cite{FuHwZa,HwFuZa}):
\[
    \tilde{V}_k(z)
    = \tilde{M}_{k,2}(z)-\tilde{M}_{k,1}(z)^2
    -z\tilde{M}'_{k,1}(z)^2,
\]
which itself also satisfies, after a straightforward calculation,
\begin{equation}\label{diff-func-vk}
    \tilde{V}_k(z)+\tilde{V}'_k(z)
    = 2\tilde{V}_{k-1}\lpa{\tfrac12z}+z\tilde{M}''_{k,1}(z)^2,
    \qquad (k\ge 1),
\end{equation}
with $\tilde{V}_0(z)= e^{-z}-(1+z)e^{-2z}$. In this form, the
original inherent cancellation is nicely integrated into the same
type of equation with an explicitly computable non-homogeneous
function, and we need only to work out the asymptotics of
$\tilde{V}_k(z)$, which will be proved to be asymptotically
equivalent to the variance of $B_{n,k}$ in the major range of
interest.

\subsection{Exact Expressions}

To justify the cancellation-free approach for computing the asymptotic
variance, we still need more explicit expressions for
$\tilde{M}_{k,2}(z)$ and $\tilde{V}_k(z)$. For that purpose, we apply
the Laplace transform on both sides of \eqref{diff-func-sm} and obtain
\[
    \mathscr{L}[\tilde{M}_{k,2}(z);s]
    = \frac{4}{s+1}\,\mathscr{L}[\tilde{M}_{k-1,2}(z);2s]
    +\frac{4}{s+1}\,\mathscr{L}[\tilde{M}_{k-1,1}(z)^2;2s],
    \qquad (k\ge 1),
\]
with $\mathscr{L}[\tilde{M}_{0,2}(z);s]= \frac1{s+1}$. This is solved by iterating the recurrence, giving
\begin{align}\label{laplace-sec-mom}
    \mathscr{L}[\tilde{M}_{k,2}(z);s]
    = \frac{4^k}{(s+1)\cdots (2^ks+1)}
    +\sum_{0\le j< k}\frac{4^{k-j}
    \mathscr{L}[\tilde{M}_{j,1}(z)^2;2^{k-j}s]}
    {(s+1)\cdots(2^{k-j-1}s+1)}.
\end{align}
From \eqref{poi-mean}, a manageable expression for the Laplace
transform of $\tilde{M}_{j,1}(z)^2$ is given by
\[
    \mathscr{L}[\tilde{M}_{j,1}(z)^2;s]
    = 4^j\sum_{0\le h,\ell\le j}\frac{(-1)^{h+\ell}
    2^{-\binom{h}{2}-\binom{\ell}{2}}}
    {Q_hQ_{j-h}Q_{\ell}Q_{j-\ell}}\cdot
    \frac{1}{s+2^{h-j}+2^{\ell-j}}.
\]
This and the partial fraction expansion \eqref{lap-Mk1} yield
\begin{align*}
    &\sum_{0\le j<k}\frac{4^{k-j}
    \mathscr{L}[\tilde{M}_{j,1}(z)^2;2^{k-j}s]}
    {(s+1)\cdots(2^{k-j-1}s+1)}\\
    &= \sum_{(j,r,h,\ell)\in\mathscr{S}}
    \frac{2^{2j+1}(-1)^{r+h+\ell}
    2^{-\binom{r}{2}-\binom{h}{2}-\binom{\ell}{2}}}
    {Q_rQ_{k-1-j-r}Q_hQ_{j-h}Q_{\ell}Q_{j-\ell}}\cdot
    \frac{1}{(s+2^{r-k+1+j})(s+2^{h-k}+2^{\ell-k})},
\end{align*}
where
\[
    \mathscr{S}
    = \{(j,r,h,\ell)\ :\ 0\le j\le k-1, 0\le r\le k-1-j,
    0\le h,\ell\le j\}.
\]
From this and the expression
\[
    \frac{1}{(s+u)(s+v)}
    = \frac{1}{v-u}\left(\frac{1}{s+u}
    -\frac{1}{s+v}\right),\qquad (u\ne v),
\]
we obtain, by term-by-term inversion,
\[
    \tilde{M}_{k,2}(z)
    = \tilde{M}_{k,1}(z)+\sum_{(j,r,h,\ell)\in\mathscr{S}}
    \frac{2^{2j+1}(-1)^{r+h+\ell}
    2^{-\binom{r}{2}-\binom{h}{2}-\binom{\ell}{2}}}
    {Q_rQ_{k-1-j-r}Q_hQ_{j-h}Q_{\ell}Q_{j-\ell}}\,
    \phi(2^{r-k+1+j},2^{h-k}+2^{\ell-k};z),
\]
where
\begin{align*}
    \phi(u,v;z)
    &= e^{-vz}\int_{0}^{z}e^{-(u-v)t}\dd{t}
    = \begin{cases}
        {\displaystyle\frac{e^{-uz}-e^{-vz}}{v-u}},
        &\text{if}\ u\ne v;\\ze^{-uz},&\text{if}\ u= v.
    \end{cases}
\end{align*}
Taking the coefficients of $z^n$ on both sides, we are led to the
exact expression for $\nu_{n,k}$
\begin{equation}\label{ex-snk}
    \nu_{n,k}
    = \mu_{n,k}+\sum_{(j,r,h,\ell)\in\mathscr{S}}
    \frac{2^{2j+1}(-1)^{r+h+\ell}
    2^{-\binom{r}{2}-\binom{h}{2}-\binom{\ell}{2}}}
    {Q_rQ_{k-1-j-r}Q_hQ_{j-h}Q_{\ell}Q_{j-\ell}}\,
    \delta(2^{r-k+1+j},2^{h-k}+2^{\ell-k};n),
\end{equation}
where
\begin{align*}
    \delta(u,v;n)
	&= n\int_0^1(1-u-(v-u)t)^{n-1}\dd t
    = \begin{cases}
        {\displaystyle\frac{(1-u)^n-(1-v)^n}{v-u}},
        &\text{if}\ u\ne v;\\n(1-u)^{n-1},&\text{if}\ u= v.
    \end{cases}
\end{align*}

Similarly, by \eqref{diff-func-vk} and the same procedure, we also
have
\begin{equation}\label{ex-var}
    \tilde{V}_k(z)
    = \sum_{(j,r,h,\ell)\in\mathscr{V}}
    \frac{2^{k-j}(-1)^{r+h+\ell}
    2^{-\binom{r}{2}-\binom{h}{2}-\binom{\ell}{2}+2h+2\ell}}
    {Q_rQ_{k-j-r}Q_hQ_{j-h}Q_{\ell}Q_{j-\ell}}\,
    \varphi\lpa{2^{j+r},2^h+2^\ell;2^{-k}z},
\end{equation}
where
\[
    \mathscr{V}
    = \{(j,r,h,\ell)\ :\ 0\le j\le k,
    0\le r\le k-j, 0\le h,\ell\le j\},
\]
and $\varphi(u,v;z)$ is defined in \eqref{varphi}. Note that
the equality $2^{j+r}= 2^{h}+2^{\ell}$ occurs if and only if
$(j,r,h,\ell)$ belongs to the set
\[
    \{(j,r,h,\ell)\ :\ 1\le j\le k, r= 0,
    h= \ell= j-1\quad\text{or}\quad 0\le j<k, r= 1, h= \ell= j\},
\]
and the corresponding terms in \eqref{ex-var} are given by
\[
    \sum_{1\le j\le k}\frac{2^{-k-j}
    2^{-2\binom{j-1}{2}+4(j-1)}}{Q_{k-j}Q_1^2Q_{j-1}^2}
    \cdot\frac{z^2}{2}\,e^{-2^{j-k}z}-
    \sum_{0\le j<k}\frac{2^{-k-j}
    2^{-2\binom{j}{2}+4j}}{Q_{k-j-1}Q_1Q_j^2}
    \cdot\frac{z^2}{2}\,e^{-2^{j+1-k}z},
\]
which becomes zero since $Q_1= \tfrac12$. Hence, the equality part in
the definition \eqref{varphi} of $\varphi(u,v;z)$ may be ignored.

\subsection{Asymptotics of $\mathrm{Var}(B_{n,k})$}

The range where $2^{k}ne^{-2^{-k}n}\to 0$ can be treated
elementarily, as in the case of the mean. In this range of $k$ (and
even in the wider range where $2^{-k}n\to\infty$), we have
\begin{equation}\label{asymp-var-simp}
    \mathrm{Var}(B_{n,k})
    \sim{\frac{2^k}{Q_k}\bigl(1-2^{-k}\bigr)^n},
\end{equation}
uniformly in $k$. To prove this, we use \eqref{ex-snk} and
begin with the estimate
\[
    \delta(2^{r-k+1+j},2^{h-k}+2^{\ell-k};n)
    = \mathcal{O}\lpa{n\lpa{1-2^{1-k}}^n},
\]
where the implied constant is absolute in $n$ and in $k$.
Substituting this into \eqref{ex-snk} yields
\begin{align*}
    \nu_{n,k}
    &= \mu_{n,k}+\mathcal{O}\llpa{n\lpa{1-2^{1-k}}^{n}
    \sum_{(j,r,h,\ell)\in\mathscr{S}}2^{2j}
    2^{-\binom{r}{2}-\binom{h}{2}-\binom{\ell}{2}}}\\
    &= \mu_{n,k}+\mathcal{O}\lpa{4^kn\lpa{1-2^{1-k}}^{n}}.
\end{align*}
By the asymptotic estimate \eqref{asymp-mean-simp} for $\mu_{n,k}$,
we then have
\[
    \nu_{n,k}
    = \frac{2^k}{Q_k}\lpa{1-2^{-k}}^n
    \llpa{1+\mathcal{O}\llpa{e^{-
    \frac{n}{2^k-1}}+2^kn\frac{\lpa{1-2^{1-k}}^{n}}
    {\lpa{1-2^{-k}}^{n}}}}.
\]
From this estimate and \eqref{two-k-ratio}, we see that
\[
    \nu_{n,k}\sim\frac{2^k}{Q_k}\lpa{1-2^{-k}}^n
    \sim\mu_{n,k},
\]
because $2^{k}ne^{-2^{-k}n}\to 0$. Also, $\mu_{n,k}= o(1)$ in this
range. Thus, we get $\mu_{n,k}^2 = o(\mu_{n,k})$ and then
\eqref{asymp-var-simp}. Note that it is possible to extend slightly
the range to $2^{k} e^{-2^{-k}n}\to0$ because there is only one term
containing the factor $n(1-2^{1-k})^n$ in the sum \eqref{ex-snk}
(which is when $j= r= h= l= 0$), and the contribution of all other
terms is bounded above by $\mathcal{O}(4^k(1-2^{1-k})^n)$. Moreover,
that \eqref{asymp-var-simp} holds in the wider range $2^{-k}n
\to\infty$ follows from a refinement of the expansion of Theorem
\ref{thm-var}, which can in turn be obtained by the analytic method
below.

On the other hand, complex analytic tools apply in a wider range. In
contrast to the mean, however, we do not prove an identity for
$\tilde{V}_k(z)$ (compare with Lemma~\ref{id-poi-mean} and see
Remark~\ref{id-var}) but we directly prove an asymptotic result
similar to the one for the mean in \eqref{exp-Mk1}.

\begin{lmm}
For $\vert\arg(z)\vert \le \varepsilon$ and $k\ge 0$, $\tilde{V}_k(z)$
satisfies the expansion
\begin{equation}\label{result-pois-mod}
	\tilde{V}_k(z)= 2^kG\lpa{2^{-k}z}+\mathcal{O}(1),
\end{equation}
where $G(z)$ is defined in Theorem~\ref{thm-var}.
\end{lmm}
\pf Similar to Lemma~\ref{id-poi-mean}, we first consider
$Q_{\infty}/Q_{k-j-r}$, which satisfies the uniform bound
\begin{equation}\label{first-order}
	\frac{Q_{\infty}}{Q_{k-j-r}}
	= \prod_{\ell\ge 1}\lpa{1-2^{j+r-k-\ell}}
	= 1+\mathcal{O}\lpa{2^{j+r-k}}.
\end{equation}
Here the product also makes sense for $j+r>k$ where it becomes zero
and thus the bound also holds in this case. Substituting this into
\eqref{ex-var} gives
\begin{align*}
	\tilde{V}_k(z)
	= 2^kG\lpa{2^{-k}z}
	+\sum_{j,r\ge 0}\sum_{0 \le h,\ell \le j}
	\frac{2^{-j-\binom{r}{2}-\binom{h}{2}
	-\binom{\ell}{2}+2h+2\ell}}
	{Q_{\infty}Q_rQ_hQ_{j-h}Q_{\ell}Q_{j-\ell}}
	\,\mathcal{O}\lpa{2^{j+r}
	\varphi\lpa{2^{j+r},2^{h}+2^{\ell};2^{-k}z}}.
\end{align*}

To estimate the double sum, we split the summation range into two:
(\emph{i}) $h,\ell \le \tr{j/2}$ and (\emph{ii}) either $h>\tr{j/2}$
or $\ell>\tr{j/2}$, and denote the resulting sums by $E_1(z)$ and
$E_2(z)$, respectively. By the estimates
\[
	2^{j+r}\varphi
	\lpa{2^{j+r},2^{h}+2^{\ell};2^{-k}z}
    = \begin{cases}
        \mathcal{O}(1), &\text{if }h,\ell \le \tr{j/2},\\
        \mathcal{O}(2^{j+r}), &\text{if }h>\tr{j/2}
        \text{ or }\ell>\tr{j/2},
    \end{cases}
\]
where the implied constants are both absolute, we have
\[
	E_1(z)
	= \mathcal{O}\llpa{\sum_{j,r\ge 0}
	\sum_{0 \le h,\ell \le j/2}
	2^{-j-\binom{r}{2}-\binom{h}{2}-\binom{\ell}{2}+2h+2\ell}}
	= \mathcal{O}(1),
\]
and
\begin{align*}
	E_2(z) &= \mathcal{O}\left(\sum_{j,r\ge 0}
	\sum_{\substack{0 \le h,\ell \le j\\
	h>j/2\ \text{or}\ \ell>j/2}}
	2^{-\binom{r}{2}-\binom{h}{2}
	-\binom{\ell}{2}+r+2h+2\ell}\right)\\
	&= \mathcal{O}\llpa{\sum_{j\ge 1}j^{-1}
	2^{-\frac{1}{8}j^2+\frac{5}{4}j}}
	= \mathcal{O}(1).
\end{align*}
This proves the claimed expansion.\qed

\begin{Rem}\label{id-var}
Comparing with Lemma~\ref{id-poi-mean}, it would be natural to derive
an identity for $\tilde{V}_k(z)$ in a way similar to that for
$\tilde{M}_{k,1}(z)$ (see Lemma~\ref{id-poi-mean}) by replacing the
first order asymptotics \eqref{first-order} by the full expansion
\[
    \frac{Q_{\infty}}{Q_{k-j-r}}
    = \prod_{\ell\ge 1}\lpa{1-2^{j+r-k-\ell}}
    = \sum_{m\ge 0}\frac{(-1)^m2^{-\binom{m+1}{2}}}
    {Q_m}\,2^{(j+r-k)m},
\]
which is zero for $j+r>k$. However, doing so yields an expression
that is no more absolutely convergent, as pointed out by one referee.

Nevertheless, ignoring the convergence issue and carrying out all
computations formally, we can expand $\tilde{V}_k(z)$ as
\begin{equation}\label{ex-var-2}
     \tilde{V}_k(z)
    = 2^k\sum_{m\ge 0}\frac{2^{-\binom{m+1}{2}-mk}}{Q_m}\,
    H_m(2^{-k}z),
\end{equation}
where $H_m(z)$ are suitable functions. Then, a formal calculation of
the Laplace transform gives (after a lengthy computation)
\begin{equation}\label{lap-H}
    \mathscr{L}\left[H_m(z);s\right]
    = s^m\sum_{j\ge 0}4^{-j}
    \frac{\tilde{g}_j^{*}\lpa{2^{-j}s}}{Q\lpa{-2^{1-j}s}},
\end{equation}
where $\tilde{g}_j^{*}(s)= \mathscr{L}[z(\tilde{M}''_{j,1}(z))^2;s]$.
Similarly,
\begin{equation}\label{lap-G-1}
    \mathscr{L}\left[G^{(m)}(z);s\right]
    = s^m\sum_{j\ge 0}4^{-j}
    \frac{\tilde{g}_j^{*}\lpa{2^{-j}s}}{Q\lpa{-2^{1-j}s}},
\end{equation}
where this relation indeed holds not just formally but also in a
rigorous analytic sense; see Appendix~\ref{A:B} for a proof.
This will be needed in the next subsection.

Now, \eqref{lap-H} and \eqref{lap-G-1} suggest that $H_m(z)=
G^{(m)}(z)$ for $m\ge 0$, which in turn suggests the validity of the
following identity
\[
     \tilde{V}_k(z)
     = 2^k\sum_{m\ge 0}
     \frac{2^{-\binom{m+1}{2}-km}}{Q_m}\,
     G^{(m)}\lpa{2^{-k}z}.
\]
This identity was claimed in the conference version of this paper,
but is not needed for our purposes here.
\end{Rem}

Note that \eqref{result-pois-mod} gives the version of
\eqref{Var-Bnk} under the Poisson model. From this we deduce
Theorem~\ref{Var-Bnk} by the same approaches used to prove
Theorem~\ref{thm-mean}, namely, de-Poissonization techniques through
the use of JS-admissible functions.
\begin{lmm}
The functions $\tilde{M}_{k,2}(z)$ are uniformly JS-admissible. More
precisely, if $\vert z\vert\ge 1$, then for $\vert\arg(z)\vert\le \ve$
\begin{equation}\label{bound-Mk2}
    \tilde{M}_{k,2}(z)
    = \mathcal{O}\lpa{\vert z\vert^{2}},
\end{equation}
and for $\ve\le \vert\arg(z)\vert\le \pi$
\[
    e^z\tilde{M}_{k,2}(z)
    = \mathcal{O}\lpa{e^{(1-\ve')\vert z\vert}},
\]
where all implied constants in the $\mathcal{O}$-terms are absolute
for $k\ge0$.
\end{lmm}
\begin{proof}
We proved in \cite[Prop. 2.4]{HwFuZa} that if $\tilde{g}(z)$ is
JS-admissible, then $\tilde{f}(z)$ with $\tilde{f}(0)= 0$ and
\[
    \tilde{f}(z)+\tilde{f}'(z)
    = 2\tilde{f}\lpa{\tfrac12z}+\tilde{g}(z),
\]
is also JS-admissible. Since $2\tilde{M}_{k-1,1}\lpa{\tfrac12z}^2$ is
uniformly JS-admissible (by Lemma \ref{JS-Mk1}), the same property
holds for $\tilde{M}_{k,2}(z)$ by the same proof of \cite[Prop.\
2.4]{HwFuZa}.
\end{proof}

We are now ready to prove Theorem~\ref{thm-var}. Since
$\tilde{M}_{k,2}\in\JS$, we have the expansion
\[
    \nu_{n,k}= \sum_{0\le j\le 2}
    \frac{\tilde{M}_{k,2}^{(j)}(n)}{j!}\, \tau_j(n)
    + \mathcal{O}\lpa{1},
\]
where we retain the terms from \eqref{pc} with $j\le 2$ so as to
guarantee that the error term is $\mathcal{O}\lpa{1}$ (here we used
estimates for the derivative of $\tilde{M}_{k,2}(z)$ which follow
from \eqref{bound-Mk2} and Ritt's theorem to bound the error term).
For $\mu_{n,k}$, we need an expansion with an error up to
$\mathcal{O}\lpa{n^{-1}}$:
\begin{align*}
    \mu_{n,k} &= \sum_{0\le j\le 2}
    \frac{\tilde{M}_{k,1}^{(j)}(n)}{j!}\, \tau_j(n)
    + \mathcal{O}\lpa{n^{-1}},
\end{align*}
so that $\mu_{n,k}^2$ is correct up to an error of order
$\mathcal{O}\lpa{1}$. Then we obtain
\begin{equation}\label{exp-var}
    \mathrm{Var}(B_{n,k})
    = \nu_{n,k}-\mu_{n,k}^2
    = \tilde{V}_{k}(n)-\frac{n}{2}\tilde{V}''_k(n)+\mathcal{O}\lpa{1},
\end{equation}
where we have used the relation $\tilde{M}_{k,2}(n)= \tilde{V}_k(n)
+\tilde{M}_{k,1}(n)^2+n\tilde{M}'_{k,1}(n)^2$. Now, from
\eqref{result-pois-mod}, we have
\[
    \tilde{V}''_k(n)
	= 2^{-k}G''(2^{-k}n)+\mathcal{O}\lpa{n^{-2}}
	= \mathcal{O}\lpa{n^{-1}},
\]
where the last estimate follows from properties of $G(z)$; see the
next subsection. Plugging this and the expansion from
\eqref{result-pois-mod} into \eqref{exp-var} gives the approximation
in Theorem \ref{thm-var}.

\subsection{Asymptotics of $G(z)$}\label{sec-Gx}

We now derive the asymptotic behaviors of $G$ for small and large
$|z|$, and prove that $G(x)$ is positive for $x\in(0,\infty)$. In
particular, the asymptotic approximations of $G$ will imply
\eqref{asymp-rel-FG}.

First, the asymptotics of $G(z)$ for large $z$ follows directly from
the defining series \eqref{G}:
\[
    G(z)
    = \frac{e^{-z}}{Q_{\infty}}
	+\mathcal{O}\lpa{|z|e^{-2\Re(z)}},
\]
for $\Re(z)>0$; see \eqref{Gx-large} which is the real-valued version
of this asymptotics. Note that similar expansions can also be derived
for the derivatives of $G(z)$.

In contrast, the small-$|z|$ asymptotics of $G$ which we now examine
turns out to be very involved. For the sake of simplicity, we
restrict ourselves only to the real case since this suffices for our
applications.
\begin{prop}\label{Th7}
For each integer $m\ge 0$, $G^{(m)}$ satisfies the asymptotic estimate
\[
    G^{(m)}(x)
    \sim 2F^{(m)}(x),
\]
as $x\to0$.
\end{prop}

The proof of this proposition is long and technical and relies mostly
on the Laplace transform. Note that since the Laplace transform of
$G^{(m)}$ is just $s^m$ times the Laplace transform of $G$, it
suffices to consider only the case $m= 0$.

We start from the Laplace transform of $G(x)$, which, by
\eqref{lap-G-1}, is given by
\begin{equation}\label{lap-G}
    \mathscr{L}[G(x);s]
    = \sum_{j\ge 0} R_j(s), \quad\text{where}\quad
    R_j(s):= \frac{\tilde{g}_j^*\lpa{2^{-j}s}}
	{4^jQ\lpa{-2^{1-j}s}}.
\end{equation}
Here $\tilde{g}_j^{*}(s)= \mathscr{L}[z(\tilde{M}''_{j,1}(z))^2;s]$,
which, by \eqref{poi-mean} and a straightforward computation, has the
form
\begin{equation*}
    \tilde{g}_j^{*}(s)
    = 4^{-j}\sum_{0\le h,\ell\le j}\frac{(-1)^{h+\ell}
    2^{-\binom{h}{2}-\binom{\ell}{2}+2h+2\ell}}
    {Q_hQ_{j-h}Q_{\ell}Q_{j-\ell}}
    \cdot\frac{1}{(s+2^{h-j}+2^{\ell-j})^2};
\end{equation*}
see also Appendix~\ref{A:B}.

Surprisingly, the dominating term in
\eqref{lap-G} is $R_2(s)\sim 2/Q(-2s)$ for large $|s|$, and the hard
part of the analysis consists in showing that $\sum_{j\ne2}R_j(s) =
\mathcal{O}(|R_2(s)/s|)$.
\begin{lmm}\label{LeG1}
For large $|s|$ in the half-plane $\Re(s)>0$, $R_0$ and $R_1$ satisfy
\begin{align}\label{j01}
    R_0(s)
    = \frac{s^{-2}}{Q(-2s)}(1+\mathcal{O}(|s|^{-1}))
    \quad\text{and}\quad
    R_1(s)
    = \frac{9\,s^{-1}}{Q(-2s)}(1+\mathcal{O}(|s|^{-1})),
\end{align}
respectively, and $R_j$ with fixed $j\ge 2$ satisfies
\begin{equation}\label{eqLeG1}
    R_j(s)
    = \frac{(2j-3)!2^{\binom{j}2}}{((j-2)!)^2}\cdot
	\frac{s^{2-j}}{Q(-2s)}
    \lpa{1 + \mathcal{O}\lpa{|s|^{-1}}}.
\end{equation}
\end{lmm}
\begin{proof}
The estimates \eqref{j01} for $j= 0$ and $j= 1$ follow from the
closed-form expressions
\[
    \tilde{g}_0^*(s)
    = \frac 1{(s+2)^2} \quad\mbox{and}\quad
    \tilde{g}_1^*(2^{-1}s)
    = \frac 4{(s+2)^2}-\frac {32}{(s+3)^2}+\frac {64}{(s+4)^2},
\]
and the functional relation $Q(-2s) = (1+s)Q(-s)$. Thus we assume now
$j\ge 2$.

Since $\tilde{g}_j^*\lpa{2^{-j}s}$ is the Laplace transform of $4^j z
\tilde M_{j,1}''(2^j z)^2$, we see that the large-$|s|$ behavior of
the former is reflected from the small-$|z|$ behavior of the latter.
Starting from \eqref{lap-Mk1-prod} using Ritt's theorem for the
asymptotics of the derivatives of an analytic function (see \cite[\S
VI. 10.1]{FlSe}), we obtain successively the estimates in the
following table.

\renewcommand\arraystretch{1.4}
\begin{center}
\begin{tabular}{c|ccc}
    $f(z)$
    & $\tilde{M}_{j,1}(z)$
    & $\tilde{M}_{j,1}''(z)$
    & $4^j z \tilde M_{j,1}''(2^j z)^2$ \\ \hline
    as $|z|\sim 0$
    & $\frac{z^j}{j!2^{j(j-3)/2}}$
    & $\frac{z^{j-2}}{(j-2)!2^{j(j-3)/2}}$
    & $\frac{2^{j(j+1)} }{((j-2)!)^2}\, z^{2j-3}$
	\\ [3pt]\hline \hline
    $\mathscr{L}[f(z);s]$
    & $\frac{4^j}{\prod_{0\le \ell \le j}(2^\ell s +1)}$
    & $\frac{4^js^2}{\prod_{0\le \ell \le j}(2^\ell s +1)}$
    & $\tilde{g}_j^*\lpa{2^{-j}s}$ \\[3pt] \hline
    as $|s|\to\infty$
    & $2^{-j(j-3)/2}\, s^{-j-1}$
    & $2^{-j(j-3)/2}\, s^{-j+1}$
    & $\frac{(2j-3)! }{((j-2)!)^2}\,
	2^{j(j+1)} s^{-2j+2}$
\end{tabular}
\end{center}
where the entries in the second and the fourth rows give the
asymptotics of $f$ and its Laplace transform as $|z|\to0$ and
$|s|\to\infty$, respectively. All error terms are of the form
$1+\mathcal{O}(|z|)$ and $1+\mathcal{O}(|s|^{-1})$, respectively,

From this table and the estimate
\begin{align}
    Q(-2^{1-j}s)
    &= \frac{Q(-2s)}{(1+s)(1+2^{-1}s)
	\cdots (1+2^{-(j-1)}s) }\nonumber\\
    &   = s^{-j} Q(-2s)
	2^{\binom{j}{2}}\lpa{1+\mathcal{O}(|s|^{-1})},
	\label{eqLeG1-2}
\end{align}
for large $|s|$, we obtain \eqref{eqLeG1}.
\end{proof}
The error term in \eqref{eqLeG1} can be further refined
\[
    R_j(s)
    = \frac{(2j-3)!2^{\binom{j}2}}{((j-2)!)^2}\cdot
	\frac{s^{2-j}}{Q(-2s)}
    \Lpa{1 -\frac{7\cdot 2^j-3}{s}
	+ \mathcal{O}\lpa{4^j|s|^{-2}}},
\]
where the $O$-term being uniform for $j=O(\log|s|)$. From this, we
expect that, for large $j$, the bound
\[
	R_j(s)
	=\mathcal{O} \Lpa{\frac{\sqrt{j}\,2^{j(j+3)/2}|s|^{-j+2}}
	{|Q(-2s)|}}
\]
holds uniformly for $j=\mathcal{O}(\log|s|)$. We give a crude version
of this in the following lemma, which also incorporates the
dependence of $j$ on $|s|$.

\begin{lmm}\label{LeG2}
Uniformly for large $|s|$ in the half-plane $\Re(s)>0$,
\begin{equation}\label{eqLeG2}
    \tilde{g}_j^*\lpa{2^{-j}s}
    = \begin{cases}
    \mathcal{O}\Lpa{\frac{j|s|^5}{\sqrt{\log_2|s|}}
    \,2^{-(\log_2|s|)^2}}, &
    \mbox{if $j\ge\log_2|s|-1$}; \\
    \mathcal{O}\lpa{\sqrt{j}\,2^{j^2+3j} |s|^{-2j+2}},
    & \mbox{if $2\le j\le \log_2|s|-1$.}
\end{cases}
\end{equation}
\end{lmm}
\begin{proof}
For notational convenience, we write the Laplace transform of $f$ as
$f^*$ and that of the product $f_1f_2$ as the convolution
\[
    (f_1^{*}\star f_2^{*})(s) := (f_1\cdot f_2)^*(s) =
    \frac 1{2\pi i} \int_{\frac12s-i\infty}^{\frac12s+i\infty}
    f_1^*(t)f_2^{*}(s-t)\dd{t}.
\]
Since $\tilde{g}_j^*\lpa{2^{-j}s}= \mathscr{L}[4^jz \tilde
M_{j,1}''(2^j z)^2;s]$, we see that, by the relation
$\mathscr{L}[zf(z);s]= -f^*(s)'$,
\[
    \tilde{g}_j^*\lpa{2^{-j}s}
    = -(L_j\star L_j)'(s) = - (L_j\star L_j')(s),
\]
where, by \eqref{lap-Mk1-prod},
\begin{align}\label{Ljs}
    L_j(s)
    := \mathscr{L}[2^j\tilde M_{j,1}''(2^jz);s]
    = \frac{s^2}{\prod_{0\le \ell \le j} \lpa{1+ 2^{-\ell}s}}.
\end{align}
On the other hand, since
\[
    L_j'(s)
    = L_j(s) \llpa{ \frac 2s - \sum_{0\le \ell \le j}
    \frac 1{s+2^{\ell}}},
\]
we derive an upper bound for each of the convolutions
\[
    \Lpa{L_j\star
    \frac{L_j}{s}}(s)\qquad\text{and}\qquad
    \Lpa{L_j\star\frac{L_j}{s+2^\ell}}(s),
    \quad\text{for}\quad 0\le \ell \le j.
\]

Note that if both $f_1^{*}(s), f_2^{*}(s) = \mathcal{O}(|s|^{-1})$ for
large $|s|$, then we can change the integration path as
\[
    (f_1^{*}\star f_2^{*})(s)
    =\frac 1{2\pi i} \int_{\gamma(s)}
    f_1^{*}(t)f_2^{*}(s-t)\dd{t},
\]
where $\gamma(s) := \bigl\{ \tfrac12s\,( 1 +iv) : -\infty < v <
\infty\bigr\}$ is the symmetry line between $0$ and $s$; such a
choice implies $|t| = |s-t|$ for $t\in \gamma(s)$ and simplifies our
analysis.

If $|s|\ge 2^{j+1}$, then $2^{-\ell-1}|s|\ge1$ for $0\le \ell \le j$,
and we have, for $0\le y\le 2^j$,
\begin{align*}
    \frac 1{2\pi i} \int_{\gamma(s)}
    \frac{L_j(t)L_j(s-t)}{t+y}\dd{t}
    = \frac{(s/2)^5}{2\pi} \int_{-\infty}^{\infty}
	\frac{(1+v^2)^2}{\frac s2(1+iv)+y}
	\prod_{0\le \ell \le j}
	\frac1{1+2^{-\ell}s + 4^{-\ell-1}s^2(1+v^2)}\dd v.
\end{align*}
By the upper bound (since $2^{-\ell-1}|s|\ge1$)
\[
	\prod_{0\le \ell \le j}
	\frac1{1+2^{-\ell}s + 4^{-\ell-1}s^2(1+v^2)}
	=\mathcal{O}\left(2^{(j+1)(j+2)}
	|s|^{-2j-2}(1+v^2)^{-j-1} \right),
\]
we then obtain
\begin{align*}
	\left|\Lpa{L_j \star \frac{L_j}{s+y}}(s)\right|
	= \mathcal{O}\left(2^{j^2+3j}|s|^{-2j+2}
	\int_0^\infty (1+v^2)^{-j+1/2} \dd v\right)
	= \mathcal{O}\lpa{j^{-1/2}2^{j^2+3j} |s|^{-2j+2}}.
\end{align*}
Thus, summing over $y=0$ and $y=2^\ell$, $\ell=0,1,\dots,j$, we
have
\[
	\tilde{g}_j^*\lpa{2^{-j}s}
	= \mathcal{O}\lpa{\sqrt{j}\,2^{j^2+3j} |s|^{-2j+2}},
\]
when $|s|\ge 2^{j+1}$, which gives \eqref{eqLeG2} in the small $j$
case.
	
On the other hand, when $1\le |s|\le 2^{j+1}$, or $j\ge \log_2|s|-1$,
we have the bounds
\[
    \prod_{0\le \ell \le j} \lpa{1+ 2^{-\ell}s}
    = \frac{Q(-2s)}{\prod_{\ell>j}\lpa{1+2^{-\ell}s}}
    = \Omega(|Q(-2s)|),
\]
because
$$
    \prod_{\ell>j}\lpa{1+2^{-\ell}s}
    = \mathcal{O}\llpa{\prod_{\ell\ge j+1}
	\lpa{1+2^{-\ell+j+1}}}
	= \mathcal{O}(1).
$$
It follows that, by \eqref{Q2s-asymp}, uniformly as
$|s|\to\infty$ in $|\arg s|\le\pi-\ve$ and $j\ge\log_2|s|-1$,
\begin{align}\label{Ljt-ae}
    L_j(s)
	= \mathcal{O}\lpa{|s|^{3/2} e^{-\Re(\log s)^2/(2\log 2)}}
    = \mathcal{O}\lpa{|s|^{3/2} 2^{-\Re(\log_2 s)^2/2}}.
\end{align}
Consequently,
\begin{align*}
	\frac 1{2\pi i} \int_{\gamma(s)}
	\frac{L_j(t)L_j(s-t)}{t+y}\dd{t}
	&= \mathcal{O}\left(|s|^3
    2^{-\Re(\log_2(s/2)^2)}
    \int_0^\infty (1+v^2)^{3/2-\Re(\log_2(s/2))}
	2^{-\frac14\log_2(1+v^2)^2} \dd v\right)\\
    &= \mathcal{O}\left(\frac{|s|^5 2^{-(\log_2|s|)^2}}
    {\sqrt{\log_2|s|}}\right).
\end{align*}
The first estimate in \eqref{eqLeG2} then follows.
\end{proof}

We now derive a precise asymptotics for $\mathscr{L}[G(z);s]$ for
large $|s|$.

\begin{lmm}\label{LeG3}
The Laplace transform of $G$ satisfies
\begin{equation*}
    \mathscr{L}[G(z);s]
    = \frac 2{Q(-2s)} \lpa{1
    +\mathcal{O}\lpa{|s|^{-1}}},
\end{equation*}
as $|s|\to\infty$ in the half-plane $\Re(s)>0$.
\end{lmm}

\begin{proof}
By Lemma~\ref{LeG1}, we have ($R_j(s)$ being defined in \eqref{lap-G})
\[
    R_0(s) + R_1(s) + R_2(s)
    = \frac 2{Q(-2s)} \lpa{1 +
    \mathcal{O}\lpa{|s|^{-1}}}.
\]
For the remaining terms, we examine the factor $Q\lpa{-2^{1-j}s}$.
By \eqref{eqLeG1-2}, we have, uniformly in the half plane
$\Re(s)>0$,
\[
    Q\lpa{-2^{1-j}s}= \begin{cases}
        \Omega(1), &\mbox{if } j\ge\log_2|s|-1;  \\
        \Omega\lpa{|s|^{-j}|Q(-2s)| 2^{\binom{j}{2}}},
        & \mbox{if } 0\le j\le \log_2|s|.
    \end{cases}
\]
Now, it follows from Lemma~\ref{LeG2} that
    \begin{align*}
    \sum_{3\le j \le \log_2|s|-1} R_j(s)
    &= \mathcal{O}\llpa{\frac 1{|Q(-2s)|}
    \sum_{3\le j \le \log_2|s|-1}
    \sqrt{j}\,2^{\frac12j(j+3) -(j-2)\log_2|s| }}\\
    &= \mathcal{O}\llpa{\frac 1{|s\, Q(-2s)|}};
\end{align*}
also by Lemma~\ref{LeG2} and \eqref{Q2s-asymp}
\begin{align*}
    \sum_{j \ge \log_2|s|-1} R_j(s)
    &= \mathcal{O}\LLpa{
    \sum_{j \ge \log_2|s|-1}\frac{j|s|^5}{\sqrt{\log_2|s|}}\,
    2^{-(\log_2|s|)^2-2j}}\\
    &= \mathcal{O}\lpa{\sqrt{\log_2|s|}\,|s|^3
    2^{-(\log_2|s|)^2}}\\
    &= \mathcal{O}\left(\frac{\sqrt{\log_2|s|}\,
    |s|^{7/2}2^{-(\log_2|s|)^2/2}}
    {|Q(-2s)|}\right)\\
    &= \mathcal{O}\left(\frac 1{|s\, Q(-2s)|}\right),
\end{align*}
which completes the proof of the lemma.
\end{proof}

\noindent\emph{Proof of Proposition \ref{Th7}.} Proposition~\ref{Th7}
now follows from Lemma~\ref{LeG3} because $\frac1{Q(-2s)}$ is the
Laplace transform of $F(z)$, details being similar to the proof of
Proposition~\ref{Th5} (for the asymptotics of $F(z)$). \qed

Finally, we prove that $G(x)$ is a positive function.
\begin{prop}
The function $G(x)$ is positive on $(0,\infty)$.
\end{prop}
\begin{proof}
By \eqref{lap-G} and the inverse Laplace transform, we see that
\[
    G(x)
    = \sum_{j\ge 0}2^j\int_{0}^{x}(x-t)
    \tilde{M}''_{j,1}(2^j(x-t))^2F(2^jt)\dd{t}.
\]
Since $F(x)$ is positive on $(0,\infty)$, we then deduce that $G(x)$
is also positive on $(0,\infty)$.\end{proof}

\section{Internal Profile}
\label{int-profile}

In this section, we present the results without proofs for the
internal profile $I_{n,k}$ because all proofs used for the external
profile extend to those for the internal profile, which satisfies the
same form of recurrence as $B_{n,k}$, namely,
\begin{equation*}
    I_{n,k}
    \stackrel{d}{= }I_{J_n,k-1}+I_{n-1-J_n,k-1}^{*},
    \qquad (n,k\ge 1),
\end{equation*}
where the notation is as in \eqref{dist-rec-exprof}. The only
differences lie in the boundary conditions: $I_{0,k}= 0$ for
$k\ge 0$ and $I_{n,0}= 1$ for $n\ge 1$.

From this recurrence and the same method used in Section \ref{mean},
we can derive the following (mostly known) result for the mean; see
\cite{DrSz,Lo}.
\begin{thm}
The expected internal profile satisfies
\[
    \mathbb{E}(I_{n,k})
    = 2^kF_I\lpa{2^{-k}n}+\mathcal{O}(1),
\]
uniformly for $0\le k<n$, where $F_I(x)$ equals the following antiderivative of
$F(x)$:
\[
    F_I(x)= 1-\sum_{j\ge 0}
    \frac{(-1)^j2^{-\binom{j+1}{2}}}
    {Q_jQ_{\infty}}\,e^{-2^jx}.
\]
Moreover, $F_I(x)$ is a positive function on $\mathbb{R}^+$; see
Figure \ref{plot-FI} for a plot.
\end{thm}

From the above expression, we see that as $x\to\infty$
\[
    F_I(x)
    = 1-\frac{e^{-x}}{Q_{\infty}}
    +\mathcal{O}\left(e^{-2x}\right),
\]
the behavior of $F_I(x)$ as $x\to0$ is the same as
\eqref{saddle-point-Fr} with $m=-1$ (by the same analysis given
there).

\begin{figure}[ht]
\begin{center}
\includegraphics[height= 5cm,width= 6cm]{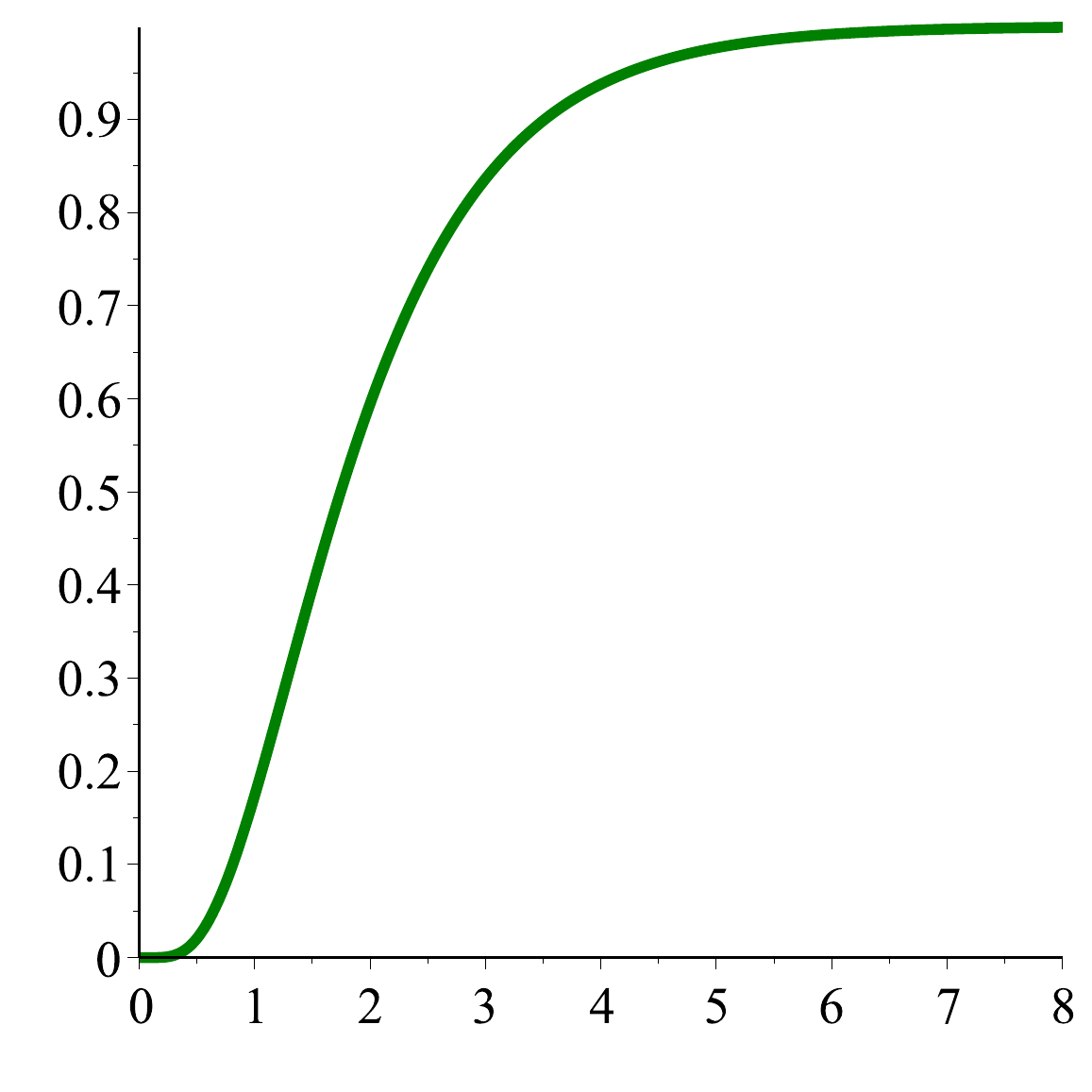}
\quad
\includegraphics[height= 5cm,width= 6cm]{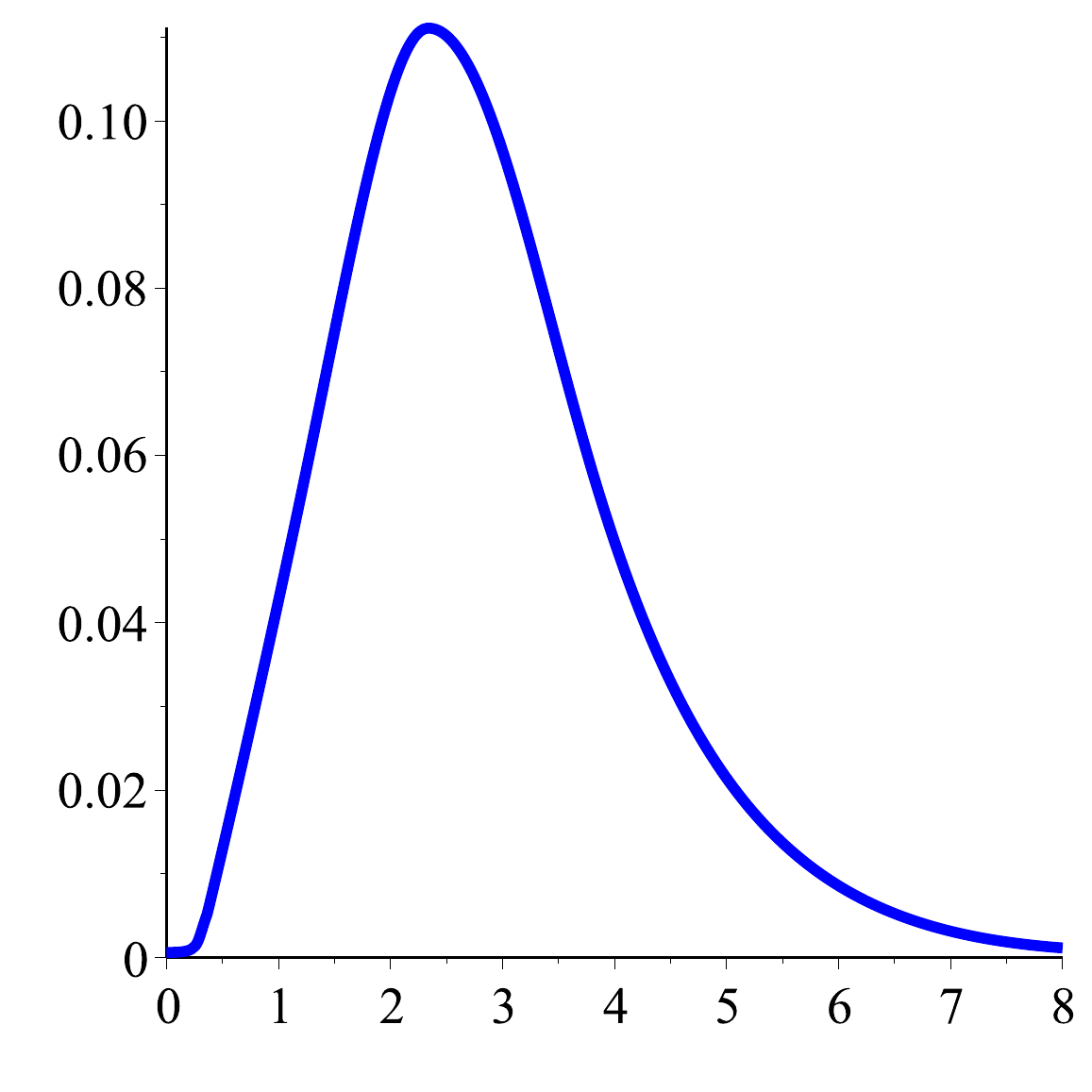}
\end{center}
\vspace*{-0.25cm}
\caption{\emph{The functions $F_I$ (left) and $G_I$ (right).}}
\label{plot-FI}
\end{figure}

Similarly, the same approaches in Section \ref{var} leads to the
following asymptotic expansion for the variance.
\begin{thm}
The variance of the internal profile satisfies
\[
    \mathrm{Var}(I_{n,k})
    = 2^kG_I\lpa{2^{-k}n}+\mathcal{O}(1),
\]
uniformly for $0\le k<n$, where $G_I(x)$ is positive on
$\mathbb{R}^+$ defined by
\[
    G_I(x)
    = \sum_{j,r\ge0}\sum_{0\le h,\ell\le j}
	\frac{(-1)^{r+h+\ell}
    2^{-j-\binom{r}{2}-\binom{h}{2}-\binom{\ell}{2}+h+\ell}}
    {Q_{\infty}Q_rQ_{h}Q_{j-h}
    Q_{\ell}Q_{j-\ell}}\,\varphi(2^{j+r},2^h+2^l;x),
\]
where $\varphi$ is defined in \eqref{varphi}.
\end{thm}
Note that the only difference between $G(x)$ and $G_I(x)$ is that the
exponent $2h+2\ell$ in the series definition \eqref{G} of $G(x)$ is
replaced by $h+\ell$; see Figure \ref{plot-FI}.

\begin{prop} The function $G_I(x)$ satisfies
\[
    G_I(x) \sim
	\begin{cases}\displaystyle
		\frac{e^{-x}}{Q_{\infty}},
		&\text{ if }x\to\infty;\\
		F_I(x), &\text{if }x\to0.
	\end{cases}
\]
\end{prop}

\begin{cor} The variance of the internal profile $I_{n,k}$ tends to
infinity if and only if there exists a positive sequence $\omega_n$
tending to infinity with $n$ such that
\begin{equation}\label{range-Ink}
    k_s+\frac{\omega_n}{\log n}
    \le k\le k_h-1-\frac{\omega_n}{\log n},
\end{equation}
where $k_s$ and $k_h$ are defined in \eqref{kskh}.

Furthermore, the expectation of the internal profile $I_{n,k}$ tends
to infinity if and only if there exists a positive sequence
$\omega_n$ tending to infinity with $n$ such that
\begin{equation}\label{range-Ink-2}
    \omega_n
    \le k\le k_h-1-\frac{\omega_n}{\log n}.
\end{equation}
\end{cor}

Note that the only difference from the central range \eqref{k0k1} for
the external profile is the additional shift of $-1$ in the upper
bound, a property implied from the fundamental relation
\[
    2I_{n,k}
    = I_{n,k+1}+B_{n,k+1}.
\]

Since $\mathrm{Var}(I_{n,k}) = o\lpa{(\mathbb{E}(I_{n,k}))^2}$ in
the range (\ref{range-Ink-2}), we thus obtain the following
concentration result.
\begin{cor}
If $\mathbb{E}(I_{n,k})\to\infty$, then the internal profile is asymptotically concentrated
around its expected value:
\[
    \frac{I_{n,k}}{\mathbb{E}(I_{n,k})}
    \stackrel{P}{\longrightarrow} 1.
\]
\end{cor}
Finally, we also expect a central limit theorem in the range
\eqref{range-Ink} where $\mathrm{Var}(I_{n,k})\to\infty$
(similarly to the external profile); this will be addressed in a
companion paper.

\section{Applications}\label{app}

In this section, we apply our results of the profiles to establish
the asymptotic two-point concentration of the height and the
saturation level in random DSTs, respectively.

\subsection{Height}

We first prove Theorem~\ref{Th:height} for the height of random DSTs.

\begin{lmm}\label{Le6.1}
For all $k$,
\begin{equation}\label{height-lr-1}
    1 - \sum_{\ell\ge 1} 2^{-\ell} \mu_{n,k+\ell}
    \le \mathbb{P}(H_n\le k)
\end{equation}
and for all $k$ with $\mu_{n,k+1}\to\infty$,
\begin{equation}\label{height-lr-2}
    \mathbb{P}(H_n\le k)
    = \mathcal{O}\left(\frac{1}{\mu_{n,{k+1}}}\right),
\end{equation}
\end{lmm}
\begin{proof}
The proof relies on the first and the second moment method.
Noting that $I_{n,k}> 0$ if and only if $H_n> k$, we have, by the
first moment method,
\[
    \mathbb{P}(H_n> k)
    = \mathbb{P}(I_{n,k}>0)
	\le \mathbb{E}(I_{n,k}).
\]
On the other hand, in view of the relation $I_{n,k} = \sum_{\ell \ge
1} 2^{-\ell}B_{n,k+\ell}$, \eqref{height-lr-1} follows from the
inequality
\[
    \mathbb{P}(H_n> k)
    \le \sum_{\ell \ge 1} 2^{-\ell} \mu_{n,k+\ell}.
\]

For the upper bound \eqref{height-lr-2}, since $B_{n,k+1} > 0$
implies that $H_{n} > k$, we see that
\[
    \mathbb{P}(H_n> k)
    \ge \mathbb{P}(B_{n,k+1} > 0).
\]
By the second moment method,
\[
    \mathbb{P}(H_n \le k)
    \le \mathbb{P}(B_{n,k+1} = 0)
    \le \frac{\sigma^2_{n,k+1}}{\mu_{n,k+1}^2}.
\]
Now, from Theorem~\ref{thm-mean} and Theorem~\ref{thm-var} and
the asymptotic growth of $G(z)$ and $H(z)$ from
Section~\ref{asymp-F} and Section~\ref{sec-Gx}, respectively, we
deduce that the variance of $B_{n,k}$ is asymptotically of the same
order as the mean:
\begin{equation}\label{rel-mean-var}
    \sigma_{n,k}^2
    = \Theta(\mu_{n,k})
\end{equation}
for $k$ with $\mu_{n,k}\to\infty$; compare with
Corollary~\ref{cor-iff}. Then \eqref{height-lr-2} follows from
\eqref{rel-mean-var}.
\end{proof}

Theorem~\ref{Th:height} is then a consequence of the following two
limit results.
\begin{lmm}\label{Le6.2}
The height of random DSTs satisfies
\[
    \lim_{n\to\infty} \mathbb{P}(H_n\le k_H+1)
    = 1\qquad\text{and}\qquad
    \lim_{n\to\infty} \mathbb{P}(H_n\le k_H-1)
    = 0.
\]
\end{lmm}
\begin{proof}
We first consider the expected value of the external profile around
the level $k_H$ (defined in \eqref{kH}) and write
\[
    k_{\ell}
    := k_{H}+\ell
	= \Bigl\lfloor\log_2 n
	+ \sqrt{2\log_2n}
    - \frac{1}{2}\log_2\log_2 n
	+ \frac{1}{\log 2}\Bigr\rfloor+\ell
	\qquad(\ell\in\mathbb{Z}).
\]
(Throughout the proof, $\ell$ is fixed unless explicitly stated otherwise.)
To prove the asymptotic concentration of the height at $k_H$ and
$k_H+1$, we also need a finer approximation of the mean. By using
more terms in the identity of Lemma~\ref{id-poi-mean} and \eqref{pc},
we obtain
\begin{equation}\label{exp-munk}
    \mu_{n,k}
	= 2^{k}F\lpa{2^{-k}n}+F'\lpa{2^{-k}n}
	+\mathcal{O}\lpa{2^{-k}n+n^{-1}}.
\end{equation}
Since $2^{-k_{\ell}}n\to 0$, we apply Proposition~\ref{Th5}
for the asymptotics of $F(2^{-k_{\!\ell}}n)$ (and its derivatives).
Note that the saddle-point equation in Proposition~\ref{Th5} has the
form $\frac{\rho}{\log\rho} = \frac{2^{k_{\!\ell}}}{n\log 2}$, or,
with $\lambda := \log_2\rho$,
\[
    \lambda-\log_2\lambda
    = \sqrt{2\log_2n}
	- \frac{1}{2}\log_2\log_2n
    + \frac{1}{\log 2}
	+ \ell
	- \theta,
\]
where $\theta= \theta_n$ denotes the fractional part of
$\log_2n+\sqrt{2\log_2 n}-\frac{1}{2}\log_2\log_2 n +\frac1{\log 2}$.
Now, by a direct bootstrapping argument, we obtain
\[
    \lambda
	= \sqrt{2\log_2 n}
    + \frac1{\log 2}
	+ \ell+\frac12-\theta
	+ \frac{2\ell+1+\frac{2}{\log 2}
    - 2\theta}{2\sqrt{2(\log 2)\log n}}
    + \mathcal{O}\Lpa{\frac{1+\ell^2}{\log n}},
\]
which also holds uniformly for $\ell=o((\log n)^{1/4})$. Substituting this
into the asymptotic expansion of $F(2^{-k_{\ell}}n)$ in
Proposition~\ref{Th5}, we have
\begin{equation}\label{eq:k-ell}
\begin{split}
    k_{\ell}\log 2+\log F\Lpa{\frac{n}{2^{k_{\ell}}}}
    &= -\sqrt{2\log_2 n}\,(\ell-1-\theta)\log 2
    - \frac{3\log\log_2n}{4}\\
	&\qquad + p(\ell-\theta)-P(\lambda)
    +\mathcal{O}\Lpa{\frac{1+\ell^2}{\sqrt{\log n}}},
\end{split}	
\end{equation}
again the $O$-term holding uniformly for $\ell=o((\log n)^{1/4})$,
where $P$ is given in \eqref{Pt} and $p(x)$ is the quadratic
polynomial:
\[
    p(x)
	= -\frac{\log 2}2\,x^2
    -(1-\log 2)x
	+ 1
	- \frac1{2\log 2}
    - \frac58\log 2
	- \frac12\log\pi.
\]
On the other hand, from \eqref{exp-munk} and the estimates (by
Proposition~\ref{Th5})
\[
    \frac{F'(2^{-k_{\ell}}n)}{F(2^{-k_{\ell}}n)}
    = \mathcal{O}(\rho)
	= \mathcal{O}\lpa{2^\lambda},
\]
we have
\begin{align}\label{height-large-k}
    \mu_{n,k_{\ell}}
	= 2^{k_{\ell}}F(2^{-k_{\ell}}n)
    \left(1+\mathcal{O}\lpa{n^{-1}\sqrt{\log n}}\right)+o(1).
\end{align}
Consequently, if $\ell\le 0$, then
\begin{equation}\label{tend-to-inf}
    \mu_{n,k_{\ell}}
	=\Omega\left(\frac{2^{\sqrt{2\log_2 n}(1+\theta)}}
	{(\log_2 n)^{3/4}}\right)+o(1)
    \to\infty,
\end{equation}
and if $\ell\ge 2$, then
\begin{equation}\label{tend-to-zero}
    \mu_{n,k_{\ell}}
	=\mathcal{O}\left(\frac{2^{-\sqrt{2\log_2 n}(1-\theta)
    }}{(\log_2 n)^{3/4}}\right)+o(1)\to 0.
\end{equation}
Now, by Lemma~\ref{Le6.1}, we show that
\begin{equation}\label{eqLe6.2}
    \sum_{\ell \ge 1} 2^{-\ell} \mu_{n,k_{\ell+1}}
    \to 0,\qquad\text{and}\qquad
    \mu_{n,k_0} \to \infty,
\end{equation}
which will then prove the lemma. The latter follows directly from
\eqref{tend-to-inf}. For the former, we use \eqref{height-large-k},
together with the expansion \eqref{eq:k-ell}, for $2\le \ell\le (\log
n)^{1/5}$, and then the boundedness of $\mu_{n,k_\ell}$ (see
Corollary~\ref{cor-central-range}) for larger $\ell$. This proves
\eqref{eqLe6.2} and the lemma.
\end{proof}
\begin{Rem}
Observe that the only missing case in \eqref{tend-to-inf} and
\eqref{tend-to-zero} is $\ell= 1$ for which we have
\[
    2^{k_1}F(2^{-k_1}n)
    = \frac{e^{\theta\sqrt{2\log_2 n}\,\log 2+\mathcal{O}(1)}}
    {(\log_2n)^{3/4}}.
\]
Thus, in this case, we have
\begin{align*}
    \mu_{n,k_1}\begin{cases}
        \to\infty, & \text{if }
        \theta\ge\frac{3\log_2\log_2 n}{4\sqrt{2\log_2 n}}
        \left(1+\frac{\omega_n}{\log_2\log_2 n}\right),\\
        \asymp 1, & \text{if }
        \theta= \frac{3\log_2\log_2 n}{4\sqrt{2\log_2 n}}
        \left(1\pm\frac{\mathcal{O}(1)}{\log_2\log_2 n}\right),\\
        \to0, & \text{if }
        \theta\le \frac{3\log_2\log_2 n}{4\sqrt{2\log_2 n}}
        \left(1-\frac{\omega_n}{\log_2\log_2 n}\right),
    \end{cases}
\end{align*}
where $\omega_n$ is any sequence tending to infinity.

We see that in almost all cases $\mu_{n,k_1}\to\infty$ and
$\mu_{n,k_1+1}\to0$, meaning that the height is in almost all cases
asymptotically equal to $k_H+1$; see also \cite{KnSz2} where this was
observed.

\end{Rem}

\subsection{Saturation Level}

Recall that the saturation level $S_n$ of a binary tree with $n$
internal nodes is defined as the maximal level with $I_{n,k} = 2^k$,
that is, up to level $S_n$ the binary tree is complete.

Define $k_S$ as
\[
    k_S
    = \left\lceil\log_2n-\log_2\log n\right\rceil,
\]
which is at the lower boundary of the central range \eqref{k0k1}.

\begin{thm}\label{Th:satlevel}
The distribution of $S_n$ is asymptotically concentrated on the two
points $k_S-1$ and $k_S$:
\[
    \mathbb{P}(S_n= k_S-1\ \text{or}\ S_n= k_S)
    \to 1,\qquad (n\to\infty).
\]
\end{thm}

\begin{Rem}
Similar to the height, the saturation level $S_n$ is concentrated at
$k_{S}-1$ in almost all cases.
\end{Rem}

The proof of Theorem~\ref{Th:satlevel} is very similar to that of
Theorem~\ref{Th:height}. The basic observation is that $S_n < k$ if
and only if $\sum_{\ell \le k} B_{n,\ell}> 0$. In particular, if
$B_{n,k} > 0$ then $S_n < k$. Hence, as above, a direct application
of the first and the second moment method implies that
\[
    1- \sum_{\ell\le k} \mu_{n,\ell}
    \le \mathbb{P}(S_n \ge k)
    \le \frac{\sigma_{n,k}^2}{ \mu_{n,k} ^2}.
\]

By similar arguments as above, Theorem~\ref{Th:satlevel} then follows
from the limit results:
\begin{align*}
    \lim_{n\to\infty}  \mathbb{P}(S_n \ge k_S-1)
    = 1\quad\text{and}\quad
    \lim_{n\to\infty}  \mathbb{P}(S_n \ge k_S+1)
    = 0.
\end{align*}
The only difference is that we now use the asymptotic expansion,
for $2^{-k}n\to\infty$,
\[
    \mu_{n,k}
    \sim \sigma_{n,k}^2
    \sim \frac{2^k}{Q_k} \lpa{1-2^{-k}}^n.
\]

\section*{Acknowledgment}

We thank the anonymous referees for a careful reading of the paper and
for helpful comments and suggestions; one of them offered many
critical, well-judged comments and suggestions to the extent of
being almost a coauthor of this paper.

\pagebreak

\appendix
\addcontentsline{toc}{section}{Appendices}
\section*{Appendices}

\section{Appendix A: Proof of Proposition~\ref{Th5}}
\label{A:A}

We give a detailed proof of Proposition~\ref{Th5}, which for
convenience is re-stated here.

\setcounter{prop}{0}

\begin{prop}
For each integer $m\ge 0$, the $m$th derivative of $F$ satisfies
\begin{equation}\label{saddle-point-Fr-A}
    F^{(m)}(z)
    = \frac{\rho^{m+\frac12+\frac1{\log 2}}}
    {\sqrt{2\pi\log_2\rho}}\,
    \exp\left(-\frac{(\log\rho)^2}{2\log 2}
    -P(\log_2\rho)\right)
    \left(1+\mathcal{O}\left(|\log\rho|^{-1}\right)\right),
\end{equation}
as $|z|\to0$ in the sector $|\arg(z)|\le \ve$, where $P(u)$ is given
in \eqref{Pt} and $\rho$ solves the equation
\begin{align*}
    \frac{\rho}{\log\rho}= \frac{1}{z\log 2},
\end{align*}
satisfying $|\rho|\to\infty$ as $|z|\to0$.
\end{prop}
\begin{proof}
Recall that
\[
    Q(s)
    := \prod_{j\ge1}\lpa{1-2^{-j}s}
	\quad\text{and}\quad
    Q_n
    := \prod_{1\le j\le n}\lpa{1-2^{-j}}
	= \frac{Q(1)}{Q(2^{-n})}.
\]
Also
\[
    F(z)
    := \sum_{j\ge0}\frac{(-1)^j2^{-\binom{j}{2}}}
	{Q_jQ_\infty}\, e^{-2^jz}.
\]
Since the Laplace transform $\mathscr{L}[F(z);s]$ of $F$ is given by
\begin{align}\label{Fz-inv-L}
    \mathscr{L}[F(z);s]
	= \prod_{j\ge0}\frac1{1+2^{-j}s}
	= \frac1{Q(-2s)}\qquad(\Re(s)>-1),
\end{align}
we have the Laplace inversion formula
\begin{equation}\label{laplace-inv-form}
    F(z)
    = \frac1{2\pi i}
	\int_{1-i\infty}^{1+i\infty}
	\frac{e^{zs}}{Q(-2s)}\,\dd s,
\end{equation}
which is valid for $z= x$ where $x>0$ is real. We are interested in
the asymptotics of $F(x)$ as $x\to 0$, which is reflected by the
large-$|s|$ asymptotics of $\mathscr{L}[F(z);s]$. Our approach relies
on the Mellin transform techniques and the saddle-point method; see
the survey paper \cite{FlGoDu} for more background tools and
applications on Mellin transform. In particular, taking logarithm on
both sides of \eqref{Fz-inv-L} (assuming that $1+2^{-j}s\ne0$), we
begin with the Mellin integral representation
\begin{align*}
    \log Q(-2s)
    = \sum_{j\ge 0}\log(1+2^{-j}s)
    = \frac{1}{2\pi i}
	\int_{-\frac12-i\infty}^{-\frac12+i\infty}
    \frac{\pi\, s^{-w}}{(1-2^{w})
    w\sin\pi w}\dd w,
\end{align*}
because the Mellin transform of $\log(1+s)$ equals
\[
    \int_{0}^{\infty}s^{w-1}\log(1+s)\dd{s}
    = \frac{\pi}{w\sin\pi w},
    \qquad (\Re(w)\in(-1,0)).
\]
Note that if $w= u+iv$ and $s= |s|e^{ib}$ with $u,v,b$ real and
$|b|\le \pi-\varepsilon$, then
\[
    \left|\frac{\pi s^{-w}}
    {(1-2^w) w\sin \pi w} \right|
	= \mathcal{O}\left(\frac{|s|^{-u}
    e^{-|v|(\pi-|b|)}}{|1-2^w||w|}\right),
\]
provided that $|w|$ stays away from the zeros of the denominator.
Thus by standard arguments, we deduce that (with $\beta :=
\frac1{2\log 2}$)
\begin{align}\label{Q2s-asymp-A}
    \log Q(-2s)
    = \beta (\log s)^2+\frac{\log s}{2}
    +P(\log_2 s)+J(s),
\end{align}
when $|\arg(s)|\le \pi-\ve$, where the periodic function $P(u)$ has
the Fourier series representation
\begin{equation}\label{Pt-A}
	\begin{split}
    P(u)
	&= \frac{\log 2}{12} +\frac{\pi^2}{6\log 2}
	- \sum_{j\ge1}
	\frac{\cos(2j\pi u)}
	{j\sinh\frac{2j\pi^2}{\log 2}},
	\end{split}
\end{equation}
which also defines an analytic function as long as
$|\Im(u)|<\pi/(\log 2)$; see Figure~\ref{Pu}. Here the remainder
$J(s)$ satisfies
\begin{align*}
	J(s) &= \frac{1}{2\pi i}
	\int_{\frac12-i\infty}^{\frac12+i\infty}
    \frac{\pi\, s^{-w}}{(1-2^{w})
    w\sin\pi w}\dd w\\
	&= \frac{1}{2\pi i}
	\int_{-\frac12-i\infty}^{-\frac12+i\infty}
    \frac{\pi\, s^{w}}{(1-2^{-w})
    w\sin\pi w}\dd w\\
	&= -\frac{1}{2\pi i}
	\int_{-\frac12-i\infty}^{-\frac12+i\infty}
    \frac{\pi\, (2s)^{w}}{(1-2^{w})
    w\sin\pi w}\dd w\\
	&= -\log Q\left(-\frac1{s}\right).
\end{align*}
\begin{figure}[!ht]
	\begin{center}
		\includegraphics[height= 3cm]{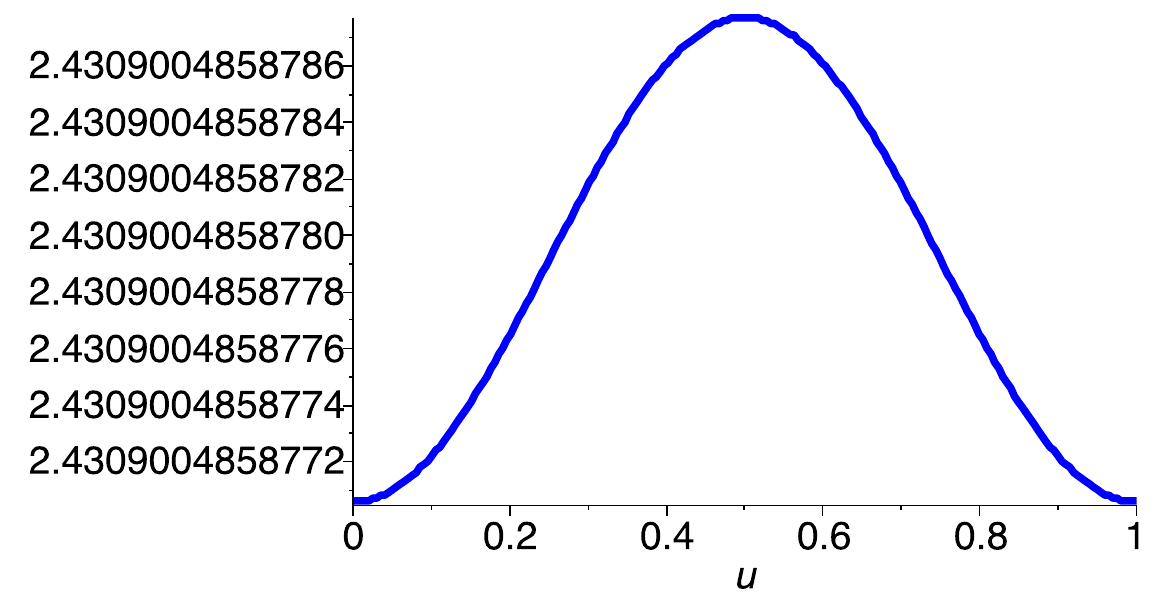}\;\;
		\includegraphics[height= 3cm]{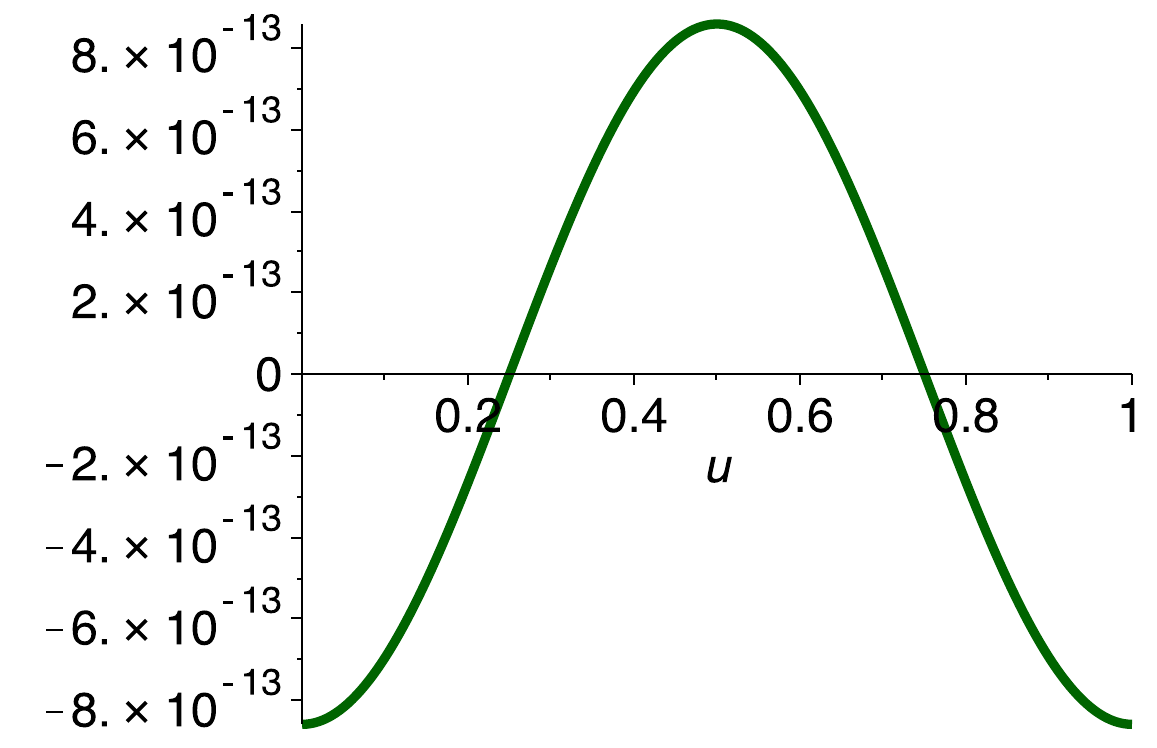}
	\end{center}
	\caption{\emph{$P(u)$ in the unit interval (left) and
	the fluctuating part of $P(u)$ (right).}}\label{Pu}
\end{figure}
We thus have the \emph{identity}:
\[
    Q(-2s)
	= \frac{\sqrt{s}\,e^{\beta(\log s)^2+P(\log_2s)}}
	{Q\lpa{-\frac1s}},
\]
or
\[
    \boxed{\prod_{j\ge0}\left(1+\frac{s}{2^j}\right)
    = \sqrt{s}\,e^{\beta(\log s)^2+P(\log_2s)}
    \prod_{j\ge1}\frac1{1+\frac{1}{2^js}}},
\]
which indeed holds, by analytic continuation, as long as
$s\in\mathbb{C}\setminus(-\infty,0]$. In particular, for large $|s|$
with $|\arg(s)|\le \pi-\ve$,
\[
    J(s)
    = -\frac1{s}+\frac{1}{6s^2}-\frac1{21s^3}
	+\frac1{60s^4}+\mathcal{O}\lpa{|s|^{-5}}.
\]

It follows, by substituting the asymptotic approximation
\eqref{Q2s-asymp}, that
\begin{align*}
    F(x)
    &= \frac1{2\pi i}\int_{1-i\infty}^{1+i\infty}
    s^{-\frac12}e^{xs-\beta(\log s)^2-P(\log_2s)}
	\left(1+\mathcal{O}\lpa{\vert s\vert^{-1}}\right)\dd s.
\end{align*}

Now, the asymptotics of $F(x)$ as $x\to 0$ is obtained by a standard
application of the saddle-point method. Therefore, we move the line
of integration to $\Re(s)= \rho$, where $\rho>0$ solves the
saddle-point equation $\frac{\log \rho}{\rho} = \frac{x}{2\beta}$.
Note that this does not change the value of the integral which is
either clear from the domain of the Laplace transform of $f(z)$ or
can also be seen directly since the integrand over the horizontal
line segments of distance $T\gg1$ from the positive real axis (and
contained in a cone with $\vert\arg(s)\vert\le \pi-\ve$) is bounded
above by
\[
	T^{-\frac12}e^{x\Re(s)-\beta(\log T)^2},
\]
implying that the integral along such lines is of order
\[
    T^{-\frac12}\exp(-\beta(\log T)^2),
\]
which decays to $0$ as $T$ tends to infinity. Thus, (with
$s\mapsto\rho(1+it)$)
\begin{align*}
    F(x) = \frac{\rho^{\frac12}e^{\rho x}}{2\pi}
	\int_{-\infty}^{\infty}
	\frac{e^{i\rho tx-\beta(\log(\rho(1+it)))^2
	-P(\log_2(\rho(1+it)))}}{\sqrt{1+it}}
	\Lpa{1+\mathcal{O}\Lpa{\frac1{\rho|1+it|}}}\dd t.
\end{align*}
By a direct iterative argument, we obtain, with $R :=
\frac{2\beta}{x}$,
\begin{align*}
    \rho
    = R\Bigg(\log R+\log\log R+\frac{\log\log R}{\log R}
    -\frac{(\log\log R)^2-2\log\log R}
    {2(\log R)^2}+\mathcal{O}
    \left(\frac{|\log\log R|^3}{|\log R|^3}\right)\Bigg).
\end{align*}
Then we split the integral into two parts:
\[
    F(x)
	= \frac{\rho^{\frac12}e^{\rho x}}{2\pi}
	\left(\int_{|t|\le t_0}+
	\int_{|t|>t_0}\right)
	\frac{e^{i\rho tx-\beta(\log(\rho(1+it)))^2
	-P(\log_2(\rho(1+it)))}}{\sqrt{1+it}}
	\Lpa{1+\mathcal{O}\Lpa{\frac1{\rho|1+it|}}}
    \dd t,
\]
where $t_0 = (\log \rho)^{-\frac25}$. Since
\begin{align*}
    \Re((\log(\rho(1+it)))^2)
	&= (\log \rho)^2+(\log\rho)\log(1+t^2)
	+\tfrac14(\log(1+t^2))^2-\arctan(t)^2
\end{align*}
is a monotonic function of $|t|$ for fixed large $\rho$, we have
\begin{align*}
	&\int_{|t|>t_0}\frac{e^{i\rho tx-\beta(\log(\rho(1+it)))^2
	-P(\log_2(\rho(1+it)))}}{\sqrt{1+it}}\dd t\\
	&\qquad = \mathcal{O}\left(e^{-\beta(\log\rho)^2}
	\int_{\log(1+t_0^2)}^\infty
	w^{-\frac12}e^{-\beta(w\log\rho
	+\frac14w^2)+w}\dd w \right)\\
	&\qquad= \mathcal{O}\left(e^{-\beta(\log\rho)^2
    -\ve(\log \rho)^{\frac15}}\right),
\end{align*}
for some $\ve>0$. Now by the local expansions
\begin{align*}
    &i\rho tx-\beta(\log(\rho(1+it)))^2\\
    &\qquad = -\beta(\log \rho)^2
	-\beta(\log\rho-1)t^2
    +\frac{\beta}{3}(2\log\rho-3)it^3
	+\mathcal{O}\lpa{t^4\log \rho},
\end{align*}
and
\[	
	e^{-P(\log_2(\rho(1+it)))}
	= e^{-P(\log_2\rho)}\left(1-\frac{P'(\log_2\rho)}
	{\log 2}\,it+\mathcal{O}(|t|^2)\right),
\]
for $|t|\le t_0$, we deduce that the integral with $|t|\le t_0$ is
asymptotic to
\begin{align*}
    F(x)&= \frac{\rho^{\frac12}e^{\rho x}}{2\pi}
	\int_{|t|\le t_0}
	\frac{e^{i\rho tx-\beta(\log(\rho(1+it)))^2
	-P(\log_2(\rho(1+it)))}}{\sqrt{1+it}}
	\Lpa{1+\mathcal{O}\Lpa{\frac1{\rho|1+it|}}} \dd t \\
    &= \frac{\rho^{\frac12}
    e^{\rho x-\beta(\log \rho)^2-P(\log_2\rho)}}
    {2\sqrt{\pi\beta\log\rho}}
    \left(1+\mathcal{O}\left((\log\rho)^{-1}\right)\right).
\end{align*}
Similarly, we also have
\begin{align*}
    F^{(m)}(x)
    &= \frac{\rho^{m+\frac12}
    e^{\rho x-\beta(\log \rho)^2-P(\log_2\rho)}}
    {2\sqrt{\pi\beta\log\rho}}
    \left(1+\mathcal{O}\left(m^2(\log\rho)^{-1}\right)\right),
\end{align*}
uniformly as $x\to0$ and $m=o(\sqrt{\log\rho})$.

\vspace*{0.2cm}
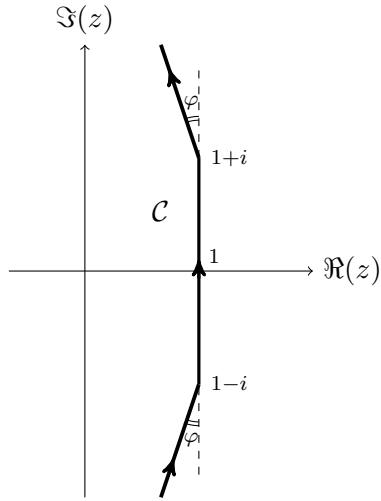
\begin{figure}[!ht]
\begin{center}
\begin{tikzpicture}
\draw[->] (0cm,0cm) -- (4cm,0cm) node[right] {$\Re(z)$};
\draw[->] (1cm,-3cm) -- (1cm,3cm) node[above] {$\Im(z)$};
\draw[line width= 0.042cm] (2.5cm,-1.5cm) -- (2.5cm,1.5cm);
\draw[line width= 0.05cm] (2.5cm,1.5cm) -- (2cm,3cm);
\draw[line width= 0.05cm] (2cm,-3cm) -- (2.5cm,-1.5cm);
\draw[thin,dashed] (2.5cm,1.5cm) -- (2.5cm,2.7cm);
\draw[thin,dashed] (2.5cm,-2.7cm) -- (2.5cm,-1.5cm);
\draw[domain= 90:110] plot ({2.5cm+0.5cm*cos(\x)},{1.5cm+0.5cm*sin(\x)});
\draw[domain= 90:111] plot ({2.5cm+0.55cm*cos(\x)},{1.5cm+0.55cm*sin(\x)});
\draw[domain= -90:-110] plot ({2.5cm+0.5cm*cos(\x)},{-1.5cm+0.5cm*sin(\x)});
\draw[domain= -90:-111] plot ({2.5cm+0.55cm*cos(\x)},{-1.5cm+0.55cm*sin(\x)});
\draw[line width= 0.042cm,->,>= stealth'] (2.5cm,-0.1cm) -- (2.5cm,0.15cm);
\draw[line width= 0.042cm,->,>= stealth'] (2.14cm,2.58cm) -- (2.11cm,2.67cm);
\draw[line width= 0.042cm,->,>= stealth'] (2.11cm,-2.67cm) -- (2.17cm,-2.49cm);

\draw (2.7cm,0.2cm) node (1) {$\scriptstyle 1$};
\draw (2.9cm,1.5cm) node (2) {$\scriptstyle 1+i$};
\draw (2.9cm,-1.5cm) node (3) {$\scriptstyle 1-i$};
\draw (2.4cm,2.23cm) node (4) {$\scriptstyle \varphi$};
\draw (2.4cm,-2.23cm) node (5) {$\scriptstyle \varphi$};
\draw (2cm,0.8cm) node (6) {$\mathcal{C}$};
\end{tikzpicture}
\end{center}
\caption{The contour of integration in the integral
representation of $F(z)$ when $z$ is complex.}\label{contour}
\end{figure}
\vspace*{0.1cm}

We now look at the situation when $z= xe^{i\theta}$ with $\theta\ne0$
and $|\theta|\le \ve$. Here, \eqref{laplace-inv-form} is no longer
valid since the integral diverges. However, by the same idea of
the Hankel contour used for extending the Gamma function, we can
deform the original integration line into the following one:
\begin{align*}
	\mathcal{C}
	&:= \{z= 1-i+e^{-i(\frac\pi2+\varphi)}u\, :\, u\ge 0\}\\
	&\qquad\cup\{z= 1+iu\, :\, -1<u<1\}
	\cup\{z= 1+i+e^{i(\frac\pi2+\varphi)}u\, :\, u\ge 0\},
\end{align*}
where $\ve<\varphi$; see Figure~\ref{contour}. Then, we use \eqref{Q2s-asymp} and make
the substitution
\begin{align*}
    F(z)
    &= \frac1{2\pi i}\int_{\mathcal{C}}
    s^{-\frac12}e^{xe^{i\theta} s-\beta(\log s)^2-P(\log_2s)}
	\left(1+\mathcal{O}\lpa{\vert s\vert^{-1}}\right)\dd s\\
	&= \frac{e^{-\frac{1}{2}i\theta}}{2\pi i}
    \int_{e^{i\theta}\mathcal{C}}
    s^{-\frac12}e^{xs-\beta(\log s-i\theta)^2
	-P\lpa{\log_2(se^{-i\theta})}}
	\left(1+\mathcal{O}\lpa{|s|^{-1}}\right)\dd s,
\end{align*}
where $e^{i\theta}\mathcal{C}$ denotes the image of $\mathcal{C}$
under the mapping $s\mapsto e^{i\theta}s$. Note that the solution to
the saddle-point equation
\[
    \frac{\log \rho(z)}{\rho(z)}
	= \frac{z}{2\beta} = \frac{e^{i\theta}}{x}
\]
where $R := \frac{2\beta}{x}$, satisfies asymptotically for small $x$
\[
    \rho(z) = Re^{-i\theta}
	\bigg(\log R+\log\log R-i\theta
	+\frac{\log\log R-i\theta}{\log R}
    +\mathcal{O}\left(\frac{|\log\log R|^2}
	{|\log R|^2}\right)\bigg).
\]
In particular ($\rho = \rho(|z|)$),
\begin{align}\label{rho-z}
    \rho(z) = \rho\,e^{-i\theta}
	\left(1-\frac{i\theta}{\log R}
	+\frac{(\log \log R-1)i\theta}{(\log R)^2}
	+\mathcal{O}\left(\frac{|\log\log R|^2}{|\log R|^3}\right)\right).
\end{align}
Since $|\theta|\le \ve$, we now deform the integration contour again
into the vertical line $\Re(s)= \rho$ (which again does not change the
value of the integral as can be seen by a similar argument as above)
and proceed as before:
\begin{align}\label{Fz-complex}
	F(z) &=
	\frac{e^{-\frac12i\theta}}{2\pi i}
	\left(\int_{\substack{s= \rho\cdot(1+it)\\
	|t|\le t_0}}+\int_{\substack{s= \rho\cdot(1+it)\\
	|t|>t_0}}\right) \nonumber\\
	&\qquad s^{-\frac12}e^{xs-\beta(\log s-i\theta)^2
	-P(\log_2s-i\theta/(\log 2))}
    \left(1+\mathcal{O}\lpa{|s|^{-1}}\right)\dd s.
\end{align}
By the local expansion
\begin{align*}
    x\rho it-\beta(\log(\rho\cdot(1+it))-i\theta)^2
    &= -\beta(\log\rho-i\theta)^2-2\beta\theta t
    -\beta(\log\rho-1-i\theta)t^2\\
    &\quad+\frac\beta3(2\log\rho-3-2i\theta)it^3+
    \mathcal{O}\lpa{(\log\rho)t^4},
\end{align*}
and the relations $(a\in\mathbb{R}, b>0)$
\begin{align*}
    \frac1{2\pi}\int_{-\infty}^{\infty}t^m
    e^{-at-bt^2}\dd t
    = \frac{(-1)^m e^{\frac{a^2}{4b}}}{2^{m+1}\sqrt{\pi b}}
	\sum_{0\le \ell\le \lfloor{\frac12m}\rfloor}
	\frac{m! a^{m-2\ell}}
	{\ell!(m-2\ell)!b^{m-\ell}}
	\qquad(m= 0,1,\dots),
\end{align*}
we deduce that the first integral on the right-hand side of \eqref{Fz-complex} is
asymptotic to
\begin{align*}
    &\frac{\rho^{\frac12}
    e^{-\frac12 i\theta-P\lpa{\log_2(\rho e^{-i\theta})}
    +x\rho-\beta(\log\rho-i\theta)^2-
    \frac{\beta^2\theta^2}{\log\rho-1-i\theta}}}
    {2\sqrt{\pi\beta(\log\rho-1-i\theta)}}
    \left(1+\mathcal{O}\lpa{(\log\rho)^{-1}}\right)\\
    &\qquad = \frac{\rho^{\frac12}
    e^{-\frac12 i\theta-P\lpa{\log_2(\rho e^{-i\theta})}
    +x\rho-\beta(\log\rho-i\theta)^2}}
    {2\sqrt{\pi\beta\log\rho}}
    \left(1+\mathcal{O}\lpa{(\log\rho)^{-1}}\right).
\end{align*}
By \eqref{rho-z}, the right-hand side is asymptotic to
\[
    \frac{\rho(z)^{\frac12}e^{-P(\log_2\rho(z))
    +z\rho(z)-\beta(\log\rho(z))^2}}
    {2\sqrt{\pi\beta\log\rho(z)}}
    \left(1+\mathcal{O}\lpa{|\log\rho(z)|^{-1}}\right).
\]
It remains to prove the smallness of the other integral in
\eqref{Fz-complex}, which is bounded above by
\begin{align*}
    &\int_{\substack{s= \rho\cdot(1+it)\\
	|t|>t_0}} s^{-\frac12}e^{xs-\beta(\log s-i\theta)^2
	-P(\log_2(se^{-i\theta}))}\dd s\\
    &\qquad = \mathcal{O}\left(\rho^{\frac12}e^{x\rho}
    \int_{t_0}^\infty (1+t^2)^{-\frac14}
    e^{-\beta((\log\rho+\frac12\log(1+t^2))^2-
    (\theta-\arctan(t))^2)}\dd t\right).
\end{align*}
The factor $(\theta-\arctan(t))^2$ being bounded for $t$ in the range
of integration, we obtain
\[
    \mathcal{O}\left(\rho^{\frac12}e^{x\rho}\int_{t_0}^\infty
    (1+t^2)^{-\frac14}
    e^{-\beta(\log\rho+\frac12\log(1+t^2))^2}\dd t\right)
    = \mathcal{O}\left(\rho^{\frac12}e^{x\rho-\beta(\log\rho)^2
    -\ve(\log\rho)^{\frac15}} \right),
\]
which, by \eqref{rho-z}, is majorized by
\[
    \mathcal{O}\left(|\rho(z)|^{\frac12}e^{\Re(z\rho(z)
    -\beta(\log\rho(z))^2)
    -\ve|\log\rho(z)|^{\frac15}} \right).
\]
We thus obtain the approximation
\[
    F(z) = \frac{\rho(z)^{\frac12}
    e^{z\rho(z) -\beta(\log \rho(z))^2-P(\log_2\rho(z))}}
    {2\sqrt{\pi\beta\log\rho(z)}}
    \left(1+\mathcal{O}\left(|\log\rho(z)|^{-1}\right)\right),
\]
uniformly as $|z|\to0$ in the sector $|\arg(z)|\le \ve$. The proof
for the $m$th derivative of $F(z)$ is similar as above.
\end{proof}

\section{Appendix B: Proof of \eqref{lap-G-1}}
\label{A:B}

It suffices to prove \eqref{lap-G-1} for $m=0$ because the general
statement follows by well-known properties of the Laplace transform.

First, by straightforward bounds, we have $G(x)= \mathcal{O}(1)$ for
$x\ge 0$, and thus the Laplace transform $\mathscr{L}[G(x);s]$ exists
for $\Re(s)>0$. Moreover, for the computation of this Laplace
transform, we can interchange the integral and the series, and obtain
\[
	\mathscr{L}[G(x);s]
	= \sum_{j,r\ge 0}\sum_{0\le h,\ell\le j}
	\frac{(-1)^{r+h+\ell}2^{-j-\binom{r}{2}
	-\binom{h}{2}-\binom{\ell}{2}+2h+2\ell}}
	{Q_{\infty}Q_rQ_hQ_{j-h}Q_{\ell}Q_{j-\ell}}
	\cdot\frac{1}{(s+2^h+2^{\ell})^2(s+2^{j+r})}.
\]
Now, as in the proof of Proposition~\ref{Th5},
\[
	\sum_{r\ge 0}\frac{(-1)^r2^{-\binom{r}{2}}}
	{Q_{\infty}Q_r(s+2^{j+r})}
	= \frac{1}{2^{j}Q(-2^{1-j}s)}.
\]
Plugging this into the above expression gives
\begin{equation}\label{simp-1}
	\mathscr{L}[G(x);s]
	= \sum_{j\ge 0}\frac{1}{4^{j}Q(-2^{1-j}s)}
	\sum_{0\le h,\ell\le  j}
	\frac{(-1)^{h+\ell}2^{-\binom{h}{2}
	-\binom{\ell}{2}+2h+2\ell}}
	{Q_hQ_{j-h}Q_{\ell}Q_{j-\ell}}
	\cdot\frac{1}{(s+2^h+2^{\ell})^2}.
\end{equation}
Next, by differentiating \eqref{poi-mean} twice,
\[
	\tilde{M}''_{k,1}(z)
	= 2^{-k}\sum_{0\le \ell\le k}
	\frac{(-1)^{\ell}2^{-\binom{\ell}{2}+2\ell}}
	{Q_{\ell}Q_{k-\ell}}e^{-2^{\ell-k}z}
\]
and thus
\begin{align*}
	\tilde{g}_j^{*}(s)
	= \mathscr{L}[z(\tilde{M}''_{j,1}(z))^2;s]
	&= 4^{-j}\sum_{0\le h,\ell\le  j}
	\frac{(-1)^{h+\ell}2^{-\binom{h}{2}
	-\binom{\ell}{2}+2h+2\ell}}
	{Q_{h}Q_{j-h}Q_{\ell}Q_{j-\ell}}
	\cdot\frac{1}{(s+2^{h-j}+2^{\ell-j})^2}\\
	&= \sum_{0\le h,\ell\le  j}
	\frac{(-1)^{h+\ell}
	2^{-\binom{h}{2}-\binom{\ell}{2}+2h+2\ell}}
	{Q_{h}Q_{j-h}Q_{\ell}Q_{j-\ell}}
	\cdot\frac{1}{(2^j s+2^{h}+2^{\ell})^2}.
\end{align*}
Finally, substituting this into \eqref{simp-1} gives
\[
	\mathscr{L}[G(x);s]
	= \sum_{j\ge 0}4^{-j}
	\frac{\tilde{g}_j^{*}(2^{-j}s)}
	{Q(-2^{1-j}s)}\qquad (\Re(s)>0),
\]
as claimed.
\end{document}